\documentclass[a4paper]{article}
\usepackage[T1]{fontenc}

\usepackage[english]{babel}
\usepackage[latin1]{inputenc}
\usepackage{amsmath,amssymb,amsthm,xspace,enumerate,url}
\renewcommand{\phi}{\varphi}

\title{Some model theory of profinite groups
}

\author{Tim Clausen\quad and\quad Katrin Tent
}
\date{\today}

\newtheorem{satz}{Theorem}[section]

\newtheorem{theorem}[satz]{Theorem}

\newtheorem{lemma}[satz]{Lemma}
\newtheorem{proposition}[satz]{Proposition}
\newtheorem{corollary}[satz]{Corollary}
\newtheorem{definition}[satz]{Definition}

\newcommand{\nc}{\newcommand}
\nc{\sa}{semialgebraic\xspace}
\nc{\el}{elementary\xspace}
\nc{\low}{lower \el}
\nc{\inv}[1]{\frac{1}{#1}}
\nc{\G}{\Gamma}
\nc{\Np}{\N_{\scriptscriptstyle >0}}
\nc{\Z}{\mathbb{Z}}
\nc{\Q}{\mathbb{Q}}
\nc{\N}{\mathbb{N}}

\nc{\Rp}{\R_{\scriptscriptstyle >0}}
\nc{\C}{\mathbb{C}}
\nc{\F}{\mathbb{F}}
\nc{\K}{\mathcal{K}}
\nc{\Kmu}{\mathcal{K_\mu}}

\nc{\Mmu}{M_\mu}
\nc{\Tmu}{T_\mu}

\nc{\grad}{\chi^}
\nc{\U}{\mathbb{U}}

\nc{\E}{\mathbb{E}}
\nc{\Epsilon}{{\Large $\epsilon$}} 
\nc{\ap}{approximable\xspace}
\nc{\e}{\mathrm{e}}
\nc{\pe}{\lhd_{p.e}}
\nc{\ii}{\,\mathrm{i}}
\nc{\Es}{\E\setminus\R_{\scriptscriptstyle\leq0}}

\DeclareMathOperator{\Aut}{Aut}

\DeclareMathOperator{\Core}{Core}

\DeclareMathOperator{\rk}{rk}

\DeclareMathOperator{\gpdim}{d}
\DeclareMathOperator{\lcm}{lcm}
\DeclareMathOperator{\VC}{VC}
\DeclareMathOperator{\Th}{Th}

\begin{document}

\maketitle
\begin{abstract}
\end{abstract}

\section*{Introduction} 
\addcontentsline{toc}{section}{Introduction}
The main purpose of these notes is to give more background and details for the results obtained in \cite{tent} which rely heavily on deep results by Lazar, Lubotzky, Mann, du Sautoy and others. 
\cite{tent} studies profinite groups as two-sorted model theoretic structures in the language $\mathcal{L}_\text{prof}$. A profinite group as  $\mathcal{L}_\text{prof}$-structure
$\mathcal{G} = (G,I)$ consists of a group $G$, a partial order $I$, and a binary relation $K \subseteq G \times I$ which encodes a neighborhood basis of the identity consisting of open subgroups. If $K$ encodes the family of all open subgroups, we call $\mathcal{G} = (G,I)$ a full profinite group. The main result in \cite{tent} shows that the theory of a full profinite group $G$ is NIP  if and only if it is NTP$_2$ if and only if $G$ has an open subgroup which is a finite direct product of compact $p$-adic analytic groups.

Their theorem yields another characterization for compact $p$-adic analytic pro-$p$ groups:
A pro-$p$ group $G$ is compact $p$-adic analytic if and only if the full profinite group $\mathcal{G} = (G,I)$ has NIP (or NTP\textsubscript{2}).
A direct proof for  will be given in Section 7.
In Section 8 we will study elementary extensions of groups as $\mathcal{L}_\text{prof}$-structures, extending the final remarks of \cite{tent}.

\section{Profinite groups and pro-$\mathcal{C}$ groups}
Profinite groups are inverse limits of inverse systems of finite groups. Viewing finite groups as topological groups with the discrete topology, profinite groups
carry the inverse limit topology. Thus, profinite groups are compact Hausdorff and totally disconnected. Conversely, every such group is profinite.

Suppose $\mathcal{C}$ is a formation of finite groups, i.e. a nonempty class of finite groups closed under isomorphism, taking quotients and subquotients (so if $G$ is a finite group, $N_1,N_2 \trianglelefteq G$, with $G/N_1, G/N_2 \in \mathcal{C}$ then $ G/(N_1 \cap N_2) \in \mathcal{C}$.
Then a pro-$\mathcal{C}$ group is the inverse limit of a surjective inverse system of groups in $\mathcal{C}$. Important examples of 
such classes $\mathcal{C}$  of finite groups are the class of all finite groups, the class of finite solvable groups, the class of finite nilpotent groups, the class of finite $p$-groups, where $p$ is a prime,  and the class of finite cyclic groups.

\medskip

Closed subsets of a profinite groups can be described in terms of open normal subgroups. In particular, the closed subgroups of a profinite groups can be approximated by open subgroups in the following sense:

\begin{proposition}[Proposition 1.2 (ii) of \cite{dixon} and Proposition 2.1.4 (d) of \cite{ribes}] \label{prop:prof_closed_subset}\label{prop:limit_closure}
	Suppose $G$ is a profinite group.
	\begin{itemize}
		\item[(a)] For $X \subseteq G$  the closure $\overline{X}$ of $X$ is given by $\overline{X}=\bigcap_{N \trianglelefteq_o G} XN$.
		\item[(b)] If $H \leq_c G$ is a closed (normal) subgroup, then $H$ is the intersection of all open (normal) subgroups containing $H$.
	\end{itemize}
\end{proposition}

\subsection{Generating sets and rank}
A topological group $G$ is topologically generated by a set $X \subseteq G$ if the subgroup generated by $X$ is dense in $G$, that is $\overline{ \langle X \rangle } = G$. We denote the minimal size of a topologically generating set by $\gpdim(G)$ and call $G$ (topologically) finitely generated if this is number is finite.

In many cases, statements about profinite groups can be obtained by pulling back the information from the finite quotients to the inverse limit. We illustrate this kind of argument by showing how  topological generating sets for profinite groups are connected to generating sets for their quotients by open normal subgroups:
\begin{proposition}[Proposition 1.5 of \cite{dixon}] \label{prop:gen_set}
	Suppose $G$ is a profinite group and $H \leq_c G$ is a closed subgroup.
	\begin{itemize}
		\item[(a)] A subset $X \subseteq H$ generates $H$ topologically if and only if $XN/N$ generates $HN/N$ for all $N \trianglelefteq_o G$.
		\item[(b)] Suppose $d$ is a positive integer and $\gpdim(HN/N) \leq d$ for all $N \trianglelefteq_o G$. Then $\gpdim(H) \leq d$.
	\end{itemize}
\end{proposition}
\begin{proof}
	(a) If $X$ generates $H$ topologically, then $XN/N$ generates $HN/N$ for all $N \trianglelefteq_o G$.	
	Now assume $XN/N$ generates $HN/N$ for all $N \trianglelefteq_o G$. 
	Then $\langle X \rangle N = G$ for all $N \trianglelefteq_o G$ and 
	thus $\overline{\langle X \rangle} = H$ by 	Proposition~\ref{prop:limit_closure}.
	
	(b) For an open normal subgroup $N \trianglelefteq_o G$, we denote by $\mathcal{Y}_N$ the set of all $d$-tuples which generate $HN/N$. By assumption $\mathcal{Y}_N$ is nonempty for all $N \trianglelefteq_o G$.
	If $N \leq M$ are two open normal subgroups of $G$, then the canonical projection $\pi_{NM}:G/N \rightarrow G/M$ maps a generating
	$d$-tuple for $HN/N$ to a generating $d$-tuple for $HM/M$. Hence we obtain an inverse system $\{ \mathcal{Y}_N, \pi_{NM} \}$
	of nonempty finite sets. The inverse limit $\varprojlim \mathcal{Y}_N$ is nonempty by compactness, so
	if $(a_1, \dots, a_d) \in \varprojlim \mathcal{Y}_N$, then $a_1N, \dots, a_dN$ generate $HN/N$ for all $N \trianglelefteq_o G$. Hence $a_1, \dots, a_d$ generate $H$ topologically by (a).
\end{proof}

Since profinite groups are compact, open subgroups have finite index. 
By considering continuous homomorphisms to the finite symmetric groups $S_m$,
we see as in the case of (abstract) finitely generated groups  that they have few open subgroups.
\begin{proposition}[Proposition 1.6 of \cite{dixon}] \label{prop:fin_gen_fin_subgps}
	Suppose $G$ is a finitely generated profinite group. Then $G$ has only finitely many open subgroups of any given finite index
and every open subgroup of $G$ is finitely generated.
\end{proposition}

Note that in general the converse of Proposition~\ref{prop:gen_set} (b) need not hold
(but it does hold in powerful pro-$p$ groups, see Theorem~\ref{t:finite_rank}). Therefore we define:

\begin{definition}[Definition 3.12 of \cite{dixon}]
	The \emph{rank} of  a profinite group $G$ is defined as
	\[ \rk (G) = \sup \{ \gpdim(H) : H \leq_o G \} \in \N \cup \{ \infty \}. \]
\end{definition}

\subsection{Order of profinite groups and Hall subgroups}
By Proposition~\ref{prop:prof_closed_subset} every closed subgroup $H$ of a profinite group $G$ is the intersection of the open subgroups of $G$ which contain $H$. This can be used to give a definition for the index of closed subgroup $H$ in $G$ that generalizes the notion of index for finite index subgroups, is well behaved with respect to the above intersection property, and allows us to perform divisibility arguments for the index of $H$ in $G$ even when it is not finite.
We follow Section 2.3 of \cite{ribes}.

\begin{definition}
	A supernatural number is a formal product $n = \prod_p p^{n(p)}$, where $p$ runs through the set of all primes and each $n(p)$ is an element of $\mathbb{N} \cup \{ \infty \}$.
\end{definition}

These numbers generalize certain aspects of natural numbers. Obviously every natural number can be viewed as a supernatural number. We can multiply two supernatural numbers by
\[ n \cdot m = \prod_p p^{n(p)+m(p)}. \]
We say $n$ divides $m$, $n|m$, if $n(p) \leq m(p)$ holds for all primes $p$.
For a set of primes $\pi$,
we call a supernatural number $\prod_p p^{n(p)}$  a $\pi$-number if $n(p) = 0$ for all $p \in \pi'$ where $\pi'$
denotes the set of all primes which are not contained in $\pi$.

Given a family $\{ n_i : i \in I \}$ of supernatural numbers, we define the least common multiple as
\[	\lcm \{ n_i: i\in I \} = \prod_p p^{n(p)}, \text{ where } n(p) = \sup \{ n_i(p) : p \text{ prime}\}.\]

\begin{definition}
	Given a profinite group $G$ and a closed subgroup $H \leq_c G$. We define the \emph{index} of $H$ in $G$ to be the supernatural number
	\[ |G:H| = \lcm \{ |G/U:HU/U| : U \trianglelefteq_o G \}. \]
	The order of $G$ is the index of the trivial subgroup in $G$, $|G| = |G:1|$.
\end{definition}

With these definitions it is not hard to see that this notion of index for closed subgroups behaves as it should. In particular we have
\begin{proposition}[Proposition 4.2.3 of \cite{ribes}] \label{prop:finite_index_divides}
	Let $G$ be a profinite group and let $H \leq G$ be a subgroup of finite index. Then $|G:H|$ divides $|G|$.
\end{proposition}
\begin{proof}
	Let $N = \bigcap_{g \in G} H^g$ be the core of $H$ in $G$. Then $N$ is a finite index 
	normal subgroup of $G$ and it suffices to show that $|G:N|$ divides $|G|$.
	Let $p$ be a prime dividing $|G:N|$ and pick
	$x \in G \setminus N$ such that $x^p \in N$.
	Then $x$ has order $p$ in $\overline{\langle x \rangle } / 
	\overline{ \langle x \rangle } \cap N$.
	Every procyclic group is a quotient of $\hat{\Z}$ and hence  
	$\overline{ \langle x \rangle } \cap N$ is an open normal subgroup of 
	$\overline{ \langle x \rangle }$. Therefore $p$ divides
	$| \overline{ \langle x \rangle } | = |\overline{ \langle x \rangle } : \overline{ \langle x \rangle } \cap N | \cdot |\overline{ \langle x \rangle } \cap N |$, and thus $p$ divides $|G| = |G:\overline{ \langle x \rangle }| \cdot |\overline{ \langle x \rangle }|$.
	
	 Now let $n$ and $m$ be maximal such that $p^n$ divides $|G:N|$ and $p^m$
	divides $|G|$. Suppose towards a contradiction that  $m < n$. We have
	\[ |G| = \lcm \{ |G:U| : U \leq_o G \}, \]
	and as $m$ is finite, there is an open subgroup $U \leq_o G$ such that $p^m$ divides $|G:U|$. Then $p$ does not divide $|U|$ because $m$ is maximal. In particular, $p$ does not divide $|U : U \cap N|$ by the first part of this proof. Hence $m$ is maximal such that $p^m$ divides $|G: U \cap N|$ contradicting our assumption that $p^n$ divides $|G:N|$. 
\end{proof}

We also give the following definition in analogy to the finite case:
\begin{definition}
	Let $G$ be a profinite group, $H$ a closed subgroup of $G$, and let $\pi$ be a set of primes.
	\begin{itemize}
		\item[(a)] $H$ is a $\pi$-group if $|H|$ is a $\pi$-number.
		\item[(b)] $H$ is a $\pi$-Hall subgroup of $G$ if $|H|$ is a $\pi$-number and $|G:H|$ is a $\pi'$-number.
	\end{itemize}
A $p$-Sylow subgroup of $G$ is  a $\pi$-Hall subgroup of $G$ where $\pi = \{p\}$ consists of a single prime number.
\end{definition}

As in the finite case, $\pi$-groups are preserved under continuous homorphisms and we have:
	\begin{lemma}\label{l:Hall} Suppose $\phi : G \rightarrow K$ is a continuous homomorphism  between profinite groups and $H \leq_c G$ is a closed subgroup of $G$. 	If $H$ is a $\pi$-Hall subgroup of $G$, then $\phi(H)$ is a $\pi$-Hall subgroup of $\phi(K)$.
	\end{lemma}

 From this observation it follows easily that
 $H$ is a $\pi$-Hall subgroup of a profinite group $G$ if and only if $HN/N$ is a $\pi$-Hall subgroup in $G/N$ for all $N \trianglelefteq_o G$.
Thus we obtain:
\begin{theorem}[Theorem 2.3.5 of \cite{ribes}]
	Suppose $\pi$ is a set of primes and $G = \varprojlim G_i$ is the inverse limit of a surjective inverse system of finite groups $\{ G_i, \phi_{ij} \}$. Assume that every $G_i$ satisfies
	\begin{itemize}
		\item[(i)] $G_i$ contains a $\pi$-Hall subgroup,
		\item[(ii)] any $\pi$-subgroup of $G_i$ is contained in a $\pi$-Hall subgroup of $G_i$,
		\item[(iii)] any two $\pi$-Hall subgroups of $G_i$ are conjugate.
	\end{itemize}
	Then $G$ satisfies the corresponding condition
	\begin{itemize}
		\item[(i')] $G$ contains a $\pi$-Hall subgroup,
		\item[(ii')] any closed $\pi$-subgroup of $G$ is contained in a $\pi$-Hall subgroup of $G$,
		\item[(iii')] any two $\pi$-Hall subgroups of $G$ are conjugate.
	\end{itemize}
\end{theorem}
As a corollary we obtain the Sylow Theorem for profinite groups:

\begin{corollary}[Corollary 2.3.6 of \cite{ribes}] \label{cor:p_sylow}
	Suppose $G$ is a profinite group and $p$ is a prime. Then the following hold:
	\begin{itemize}
		\item[(a)] G contains a $p$-Sylow subgroup,
		\item[(b)] any $p$-subgroup of $G$ is contained in a $p$-Sylow subgroup of $G$ and
		\item[(c)] any two $p$-Sylow subgroups of $G$ are conjugate.
	\end{itemize}
\end{corollary}

We also need the following analog of the finite situation:

\begin{corollary}[Proposition 2.3.8 of \cite{ribes}] \label{cor:pro_nilpotent}
	A profinite group $G$ is pronilpotent if and only if for each prime $p$, $G$ contains a unique $p$-Sylow subgroup. In that case, $G$ is the direct product of its Sylow subgroups.
\end{corollary}
\begin{proof}
	If $G$ is pronilpotent, then $G$ is the inverse limit $G = \varprojlim G_i$ of a surjective inverse system of finite nilpotent groups. Let $\phi_i: G \rightarrow G_i$ be the projection map and assume that $P_1$ and $P_2$ are $p$-Sylow subgroups of $G$. Then $\phi_i(P_1) = \phi_i(P_2)$ for all $i$ as both $\phi_i(P_1)$ and $\phi_i(P_2)$ are $p$-Sylow subgroups of $G_i$ by Lemma~\ref{l:Hall} and $G_i$ is nilpotent.
\end{proof}

Let $K$ be a closed normal subgroup of a profinite group $G$. A \emph{complement} of $K$ in $G$ is a closed subgroup $H$ of $G$ such that $G = KH$ and $K \cap H = 1$, so  $G = K \rtimes H$. In particular, if $K$ is a $\pi$-Hall subgroup of $G$ then $H \cong G/K$ is a $\pi'$-subgroup of $G$.

The following is a generalization of the Schur-Zassenhaus Lemma  for finite groups (see Theorem 7.41 and Theorem 7.42 in \cite{rotman}) to profinite groups.

\begin{proposition}[Proposition 2.3.15 of \cite{ribes}] \label{schur}
	Suppose $K$ is a closed normal Hall subgroup of a profinite group $G$. Then $K$ has a complement in $G$ and any two complements of $K$ are conjugate in $G$.
\end{proposition}
\begin{proof}
	For any open normal subgroup $N\trianglelefteq_o G$, the group $KN/N$ is a Hall subgroup 
	of the finite group $G/N$ and hence has a complement by  the Schur-Zassenhaus theorem 
	for finite groups. The set of complements of the  $KN/N$ in $G/N$ forms an inverse 
	system of finite subgroups whose limit $H$, say, is a closed subgroup of $G$ 
	and a complement for $K$ in $G$.

	Now assume $L$ is another complement of $K$ in $G$. To see that $H$ and $L$ are conjugate
	in $G$, note that $HN/N$ and $LN/N$ are both complements of $KN/N$ in $G/N$ for all 
	$N \trianglelefteq_o G$ and hence conjugate in $G/N$ by the Schur-Zassenhaus theorem 
	for finite groups. As before we obtain an inverse system of the conjugating elements
	whose limit will conjugate $L$ to $H$.
\end{proof}

\subsection{Frattini subgroup}
A proper subgroup $H < G$ is said to be maximal in $G$ if there is no subgroup $K <G$ such that $H < K <G$.
A maximal open subgroup is an open subgroup which is maximal with respect to open subgroups.

\begin{definition}
	Let $G$ be a profinite group. The Frattini subgroup $\Phi(G)$ of $G$ is the intersection of all maximal open subgroups.	
\end{definition}
Clearly, the Frattini subgroup of a profinite group is topologically characteristic. Moreover, as in the finite case it is precisely the set of \emph{nongenerators}:

\begin{proposition}[Proposition 1.9 (iii) of \cite{dixon}] \label{prop:frattini}
	Suppose $G$ is a profinite group and $X \subseteq G$ is a subset of $G$. Then $X$ generates $G$ topologically if and only if $X \cup \Phi(G)$ generates $G$ topologically.
\end{proposition}
\begin{proof}
	Clearly, if $X$ generates, then so does $X \cup \Phi(G)$. For the converse let 
	$H \leq_o G$ be an open subgroup such that $X \subseteq H$. 
	 If $H \neq G$,
	then $H$ is contained in some maximal open subgroup $M <_o G$ and therefore
	\[ \overline{\langle X \rangle} \Phi(G) \leq M \neq G. \]
\end{proof}

\subsection{Pro-$p$ groups}

Using the analogous results for finite $p$-groups and pulling back along the inverse system one proves the following:
\begin{lemma}[Lemma 2.8.7 of \cite{ribes}] \label{lem:pro_p_frattini}
	Suppose $G$ is a pro-$p$ group.
	\begin{itemize}
		\item[(a)] Every maximal open subgroup $M \leq_o G$ has index $p$ and is normal in~$G$.
		\item[(b)] The quotient $G/\Phi(G)$ is an elementary abelian profinite group and thus a vector space over $\mathbb{F}_p$.
		\item[(c)] The Frattini subgroup is given by $\Phi(G) = \overline{G^p[G,G]}$.
	\end{itemize}
\end{lemma}

\begin{corollary}[Lemma 2.8.10 of \cite{ribes}] \label{cor:pro_p_frattini_open}
	Suppose $G$ is a pro-$p$ group. Then $G$ is finitely generated if and only if $\Phi(G)$ is open in $G$.
\end{corollary}
\begin{proof}
    Since every maximal open subgroup of $G$ has index $p$, if $G$ is finitely generated, then 
    by Proposition~\ref{prop:fin_gen_fin_subgps} there are only finitely many.
	Hence $\Phi(G)$ is a finite intersection of open subgroups and thus open.
	
	If conversely, $\Phi(G)$ is open, there is a finite set $X \subseteq G$ such that the 
	finite group $G/\Phi(G)$ is generated by the image of $X$. Then 
	$G = \overline{\langle X \rangle} \Phi(G)$ and thus
	$G = \overline{\langle X \rangle }$ by Proposition~\ref{prop:frattini}.
\end{proof}

\begin{definition}[Definition 1.15 of \cite{dixon}] \label{lower_p}
	Let $G$ be a pro-$p$ group. The \textit{lower $p$-series} is given by
	\[ P_1 = G \text{ and } P_{i+1} = \overline{P_i(G)^p [P_i(G),G] } \text{ for all } i \geq 1. \]
	We will use the notation $G_i = P_i(G)$.
\end{definition}

Note that all $P_i(G)$ are topologically characteristic, that $P_2(G) = \Phi(G)$, and that $P_{i+1}(G) \geq \Phi(P_i(G))$ for all $i$ by Lemma~\ref{lem:pro_p_frattini}.

\begin{proposition}[Proposition 1.16 of \cite{dixon}]
	Suppose $G$ is a pro-$p$ group.
	\begin{itemize}
		\item[(a)] $P_i(G/K) = P_i(G)K/K$ for all $K \trianglelefteq_c G$ and all $i$.
		\item[(b)] If $G$ is finitely generated, then $\{ P_i(G) : i \geq 1 \}$ is a  basis for the open neighborhoods of $1$ in $G$.
	\end{itemize} 
\end{proposition}
\begin{proof}
	(a) Fix $K \trianglelefteq_c G$. We argue by induction on $i$. Note that $P_1(G/K) = G_1K/K$.
		Assume $P_n(G/K) = G_nK/K$ holds for some $n$. Let $M$ be the preimage of $P_{n+1}(G/K)$ under the natural projection
		$G \rightarrow G/K$. Then $M$ is a closed subgroup of $G$, $M/K = P_{n+1}(G/K)$ and $M \geq G_n^p[G_n,G]K$.
		In particular $M \geq G_{n+1}K$ and therefore $M = G_{n+1}K$.
		
	(b) We argue by induction. $G_1 = G$ is finitely generated and open in $G$. Suppose $G_n$ is finitely generated and open in $G$.
		Then $\Phi(G_n)$ is open in $G_n$ and thus $G_{n+1} \geq \Phi(G_n)$ is open in $G_n$ and hence finitely generated by Proposition~\ref{prop:fin_gen_fin_subgps}.
		
		Now let $N \trianglelefteq_o G$ be an open normal subgroup. Then $G/N$ is a finite $p$-group and therefore $P_i(G/N) = 1$
		for all sufficiently large $i$. Then $P_i(G) \leq N$ for all sufficiently large $i$ by (a).
\end{proof}

If $G$ is a finitely generated pro-$p$ group, we can simplify the definition of the Frattini subgroup (see Theorem 1.17 of \cite{dixon}) using the fact that $[G,G]$ is closed in $G$ (see Proposition 1.19 of \cite{dixon}).

\begin{proposition} \label{prop:pro_p_without_bars}
	Suppose $G$ is a finitely generated pro-$p$ group. Then $\Phi(G) = G^p[G,G]$.
\end{proposition}
\begin{proof}
	Set $G^{\{p\} } = \{ g^p : g \in G \}$. Note that $G^{\{p\}}$ is a closed subset of $G$ as it is the image of the continuous map $x \mapsto x^p$. As $G/[G,G]$ is abelian, we obtain $G^p[G,G] = G^{\{p\}}[G,G]$.
\end{proof}

\subsection{Subgroups of finite index}
Nikolov and Segal proved in \cite{nikolov1} and \cite{nikolov2} that every finite index subgroup of a finitely generated profinite group is open. For our purposes we can do with a much weaker result. We will only prove Anderson's theorem that this holds for finitely generated pronilpotent groups. We start with Serre's Theorem, which is the first step:

\begin{theorem}[Theorem 1.17 of \cite{dixon}] \label{thm:serre}
	Let $G$ be a finitely generated pro-$p$ group. Then every subgroup of finite index is open in $G$.
\end{theorem}
\begin{proof}
	By Corollary~\ref{cor:pro_p_frattini_open} a pro-$p$ group is finitely generated if and only if its Frattini-subgroup is open. Combining this fact with Proposition~\ref{prop:pro_p_without_bars}, we have
	\[ G^p [G,G] = \Phi(G) \trianglelefteq_o G. \]
	It clearly suffices to prove the theorem for normal subgroups  $N \triangleleft G$
	of finite index. We argue by induction on $|G:N|$. Hence $N$ is open in $M$ whenever $M$ is a finitely generated pro-$p$ group such that $N \leq M < G$.
	
	Set $M = G^p [G,G] N$, then $G/N$ is a finite nontrivial $p$-group. We have $M/N = \Phi(G/N) < G$ and hence $M < G$ is a proper subgroup. Moreover, $M$ is an open subgroup of $G$ as it contains $\Phi(G)$. Therefore it is finitely generated by Corollary~\ref{cor:pro_p_frattini_open} and we can apply the induction hypothesis. Now $N \leq_o M$ and $M \leq_o G$, and thus $N \leq_o G$.
\end{proof}

To extend this result to pronilpotent groups, the following observation is crucial.

\begin{proposition}[Proposition 7 of \cite{anderson}]
	Let $G$ be a prosolvable group and let $N \trianglelefteq G$ be a normal subgroup of finite index. Then $G/N$ is solvable.
\end{proposition}
\begin{proof}
	By a theorem of Hall (see, for example, \cite[Theorem 5.29]{rotman}) a finite group is solvable if it contains $p'$-Hall subgroups for all primes $p$.
	
	Let $p$ be a prime number. By  Corollary~\ref{cor:pro_solvable} there is a $p'$-Hall subgroup
	$H \leq_c G$ of $G$. We will show that $HN/N$ is a $p'$-Hall subgroup of $G/N$.
	We have
	\[ |G:N| = |G:HN| \cdot |HN:N|. \]
	Proposition~\ref{prop:finite_index_divides} implies that $|G:HN|$ divides $|G:H|$ and thus is a power of $p$. On the other hand $HN/N \cong H/(H \cap N)$.
	Hence $|HN:N|$ divides $|H|$ by Proposition~\ref{prop:finite_index_divides}. Therefore $p$ does not divide $|HN:N|$, and hence $HN/N$ is a $p'$-Hall subgroup of $G/N$.
\end{proof}

\begin{theorem}[Corollary following Theorem 3 of \cite{anderson}] \label{anderson}
	Let $G$ be a finitely generated pronilpotent group. Then every subgroup of finite index is open.
\end{theorem}
\begin{proof}
	Fix  a  subgroup $N \leq G$ of finite index. We may assume that $N$ is normal in $G$ 
	by replacing $N$
	by $\Core _G(N)=\bigcap_{g\in G} N^g$. Since  
	$G$ is pronilpotent, it is prosolvable and so $G/N$ is a finite solvable group. Therefore $G/N$ admits a subnormal series with cyclic quotients of prime order. Hence we may assume that $p = |G:N|$ is prime.
	
	Let $q \neq p$ be a different prime and let $S_q$ be the unique $q$-Sylow subgroup of $G$.
	Then $|S_q : S_q \cap N|$ divides $|S_q|$ by Proposition~\ref{prop:finite_index_divides} and as
	\[ S_q / (S_q \cap N) \cong S_q N / N \leq G /N, \]
	we see that $|S_q : S_q \cap N|$ divides $p$. Therefore $|S_q : S_q \cap N| = 1$ and thus $S_q \leq N$.
	
	As $G$ is the product of its Sylow subgroups, it is enough to show that $S_p \cap N$ is open in $S_p$. As $G$ is finitely generated, $S_p$ has only finitely many open subgroups of index $p$. Therefore $\Phi(S_p)$ is open in $S_p$ and hence
	$S_p$ is finitely generated by  Corollary~\ref{cor:pro_p_frattini_open}. Now use 
	Theorem~\ref{thm:serre} to see that $S_p \cap N$ is an open subgroup of $S_p$.
\end{proof}

\subsection{Pro-Fitting subgroup}
In a finite group the Fitting subgroup is the maximal normal nilpotent subgroup. We will see that the pro-Fitting subgroup of a profinite group  is the maximal normal pronilpotent subgroup. We can define it as follows:

\begin{definition}[Definition 1.3.9 of \cite{reid}] Let $G$ be a profinite group.
	\begin{itemize}
	\item[(a)] The pro-Fitting subgroup $F(G)$ of $G$ is the closed subgroup generated by all subnormal pro-$p$ subgroups of $G$ where $p$ runs over all primes.
	\item[(b)] If $\pi$ is a set of primes then the $\pi$-core $O_\pi(G)$ of $G$ is the closed subgroup generated by all
		subnormal pro-$\pi$ subgroups of $G$.
	\end{itemize}
\end{definition}

Obviously, both $F(G)$ and $O_\pi(G)$ are topologically characteristic subgroups of $G$. The $\pi$-core of a finite groups is the unique maximal normal $\pi$-subgroup.

\begin{lemma}[Lemma 1.3.10 of \cite{reid}]
	Let $G$ be a profinite group and let $\pi$ be a set of primes. Let $R_N$ be the preimage of $O_\pi(G/N)$ under the canonical projection $G \rightarrow G/N$ for all $N \trianglelefteq_o G$. Then $O_\pi(G) = \bigcap_{N \trianglelefteq_o G} R_N$ and in particular $O_\pi(G)$ is a $\pi$-group.
\end{lemma}
\begin{proof}
	Note that $R$ is a pro-$\pi$ group as $RN/N \leq R_N/N = O_\pi(G/N)$ is a $\pi$-group for all $N \trianglelefteq_o G$.
	As $O_\pi(G)$ is generated by all subnormal pro-$\pi$ groups we have $R \leq O_\pi(G)$.
	
	On the other hand , $O_\pi(G)N/N$ is generated by subnormal $\pi$-groups and 
	therefore $O_\pi(G)N/N \leq R_N$ for all $N \trianglelefteq_o G$ and t
	herefore $O_\pi(G) \leq R$ by Proposition~\ref{prop:prof_closed_subset}.
\end{proof}

If $p$ and $q$ are different primes, then $O_p(G) \cap O_q(G) = 1$ by the above lemma. Therefore $F(G)$ is the direct product of the $O_p(G)$ and it follows easily from  Corollary~\ref{cor:pro_nilpotent} that $F(G)$ is pronilpotent. As each $O_p(G)$ is contained in $F(G)$, $F(G)$ must be the maximal normal pronilpotent subgroup.
Moreover, by Theorem 5.4.4 and Corollary 5.4.5 of \cite{reid} the following holds. 

\begin{theorem} \label{reid}
	Let $G$ be a profinite group of finite rank. Then there is an open normal subgroup $A \trianglelefteq_o G$, $F(G) \leq A \leq G$, such that $A/F(G)$ is finitely generated abelian.
\end{theorem}

\subsection{Automorphisms of profinite groups}
For a profinite group $G$ we want to consider the group $\Aut(G)$ of continuous automorphisms of $G$
again as a profinite group. 
If $\gamma$ is a continuous automorphism of $G$ and $g\in G$ is an element, we write $g^\gamma = \gamma(g)$ and $[g,\gamma] = g^{-1}g^\gamma$.
Given an open normal subgroup $N \trianglelefteq_o G$ of $G$, we consider
\[ \Gamma(N) = \{ \gamma \in \Aut(G) : [G,\gamma] \subseteq N \}. \]
Note that a continuous automorphism $\gamma$ of $G$ is in $\Gamma(N)$ if and only if it leaves $N$ invariant and acts trivially on $G/N$.
We view $\Aut(G)$ as a topological group, where the family $\{ \Gamma(N) : N \trianglelefteq_o G \}$
is a neighborhood basis of the identity.

\begin{theorem}[Proposition 4.4.3 of \cite{ribes}]
	Suppose $G$ is a profinite group and $\mathcal{U}_c$ is a  neighbourhood basis of the identity consisting of open characteristic subgroups.
	Then $\Aut(G)$ is profinite.
\end{theorem}
\begin{proof}
	Let $U \in \mathcal{U}_c$. Then $\Gamma(U)$ is the kernel of the natural homomorphism
	\[ \omega_U: \Aut(G) \rightarrow \Aut(G/U). \]
	Hence $\omega_U$ is continuous. Put $A_U = \omega_U(\Aut(G))$. Given $V \leq U$ in $\mathcal{U}_c$, there is a canonical map
	\[ \omega_{VU} : A_V \rightarrow A_U, \phi/V \mapsto \phi/U. \]
	The maps $\omega_{VU}$ are well-defined homomorphisms 
	inducing an epimorphism
	\[ \omega : \Aut(G) \twoheadrightarrow \varprojlim_{U \in \mathcal{U}_c} A_U. \]
	Now $\ker(\omega) = \bigcap_{U \in \mathcal{U}_c} \Gamma(U) = 1$. Hence $\omega$ is injective and therefore $\Aut(G) = \varprojlim_{U \in \mathcal{U}_c} A_U$ is profinite.
\end{proof}

\begin{corollary}
	Suppose $G$ is a finitely generated profinite group. Then $\Aut(G)$ is profinite.
\end{corollary}
\begin{proof}
	Let $i > 0$ and let $N_i$ be the intersection of all open subgroups of index at most $i$. Then $N_i$ is an open subgroup by Proposition~\ref{prop:fin_gen_fin_subgps} and it is obviously topological characteristic.
	Hence the family $\{ N_i: i > 0 \}$ is  a  neighbourhood basis of the identity. Now we can apply the previous theorem.
\end{proof}

\begin{lemma}[Section 5 Exercise 4 of \cite{dixon}]
	Suppose $G$ is a finite $p$-group and $H = \{ \alpha \in \Aut(G) : [G,\alpha] \subseteq \Phi(G) \}$. Then $H$ is a finite $p$-group.
\end{lemma}
\begin{proof}
	Suppose $\alpha \in H \setminus \{1\}$ has prime order $q$. We fix a generating set $\{x_1, \dots, x_d \}$ for $G$.
	Set $\Omega = \{ (u_1x_1, \dots, u_dx_d ) : u_i \in \Phi(G) \} \subseteq G^d$.
	Each set $\{u_1x_1, \dots, u_dx_d \}$ generates $G$ by Proposition~\ref{prop:frattini} 
	and hence no element of $\Omega$ is fixed under the action of $\alpha$.
	Thus any orbit has length $q$ and hence $q$ divides $ |\Phi(G)|^d$. It follows $q=p$.
\end{proof}

\begin{proposition}[Proposition 5.5 of \cite{dixon}] \label{prop:gamma_pro_p}
	Suppose $G$ is a finitely generated pro-$p$ group. Then $\Gamma(\Phi(G))$ is a pro-$p$ group.
\end{proposition}
\begin{proof}
	The subgroups $\{ \Gamma(G_n) : n \geq 2 \}$ are normal in $\Gamma( \Phi(G))$ and form a basis for the neighborhoods of $1$. It suffices to show that $\Gamma(\Phi(G))/\Gamma(G_n)$ is pro-$p$ for each $n \geq 2$.
	Notice that $\Gamma(\Phi(G))/\Gamma(G_n)$ acts faithfully on the finite $p$-group $G/G_n$ and induces the trivial action on $G/\Phi(G) = (G/G_n)/\Phi(G/G_n)$. Now apply the above lemma to see that $\Gamma(\Phi(G))/\Gamma(G_n)$
	is a finite $p$-group.
\end{proof}

A profinite group $G$ has \emph{virtually} a property $\mathcal{P}$ if it has an open normal subgroup $N \trianglelefteq_o G$ such that $N$ has property $\mathcal{P}$.

\begin{theorem}[Theorem 5.6 of \cite{dixon}] \label{aut_pro_p}
	Suppose $G$ is a finitely generated profinite group. If $G$ is virtually a pro-$p$ group then so is $\Aut(G)$.
\end{theorem}
\begin{proof}
	The group $G$ has a topological characteristic open normal pro-$p$ subgroup $H$. Note that $H$ is finitely generated and therefore $\Phi(H)$ is open and topological characteristic in $G$.
	Let $\Delta = \Gamma(\Phi(H))$ be the kernel of the action of $\Aut(G)$ an $G/\Phi(H)$. Then $\Delta \trianglelefteq_o \Aut(G)$ and we will see that $\Delta$ is pro-$p$.
	
	Consider the restriction map $\pi : \Delta \rightarrow \Aut(H)$ and set $\Xi = \ker \pi$. Given $N \trianglelefteq_o H$ we have $(\Delta \cap \Gamma(N) )^\pi \subseteq \Gamma_H(N)$ and therefore $\pi$ is continuous.
	Hence $\Xi$ is a closed normal subgroup of $\Delta$.
	
	By Lemma 5.4 of \cite{dixon}, $\Xi$ is isomorphic to a closed subgroup of $Z(H)^{(m)}$, where $m = \gpdim(G/H)$ is finite. As $H$ is pro-$p$, so is $\Xi$. The group $\Delta/\Xi$ is isomorphic to the closed subgroup $\Delta^\pi$
	of $\Aut(H)$. Note that in fact $\Delta^\pi \leq \Gamma_H(\Phi(H))$ and 
	therefore $\Delta^\pi$ is pro-$p$ by Proposition~\ref{prop:gamma_pro_p}. Hence $\Delta$ is pro-$p$.
\end{proof}

\section{Powerful groups}
In this section we study powerful pro-$p$ groups. These are pro-$p$ groups in which the subgroup generated by the $p$-th powers is large. The lower $p$-series then gives  them a canonical structure reminscent of an abelian pro-$p$ group. This analogy will be used later in order to identify certain groups as $p$-adic analytic. Throughout this section, $p$ will be a fixed prime number different from $2$. If $G$ is a $p$-group or a pro-$p$ group we will set $G_i = P_i(G)$ to be the $i$-th group in the lower $p$-series (\ref{lower_p}).

\subsection{Powerful \textit{p}-groups}
We start by studying finite $p$-groups. We follow Section 2 of \cite{dixon} to develop the theory of powerful $p$-groups. Throughout we assume $p\neq 2$ (and refer to \cite{dixon} for the general case).

\begin{definition}
	Suppose $G$ is a finite $p$-group, $p\neq 2$.
	\begin{itemize}
		\item[(a)] $G$ is \textit{powerful} if $G/G^p$ is abelian.
		\item[(b)] A subgroup $N \leq G$ is powerfully embedded in $G$ if  $[N,G] \leq N^p$ .
	\end{itemize}
\end{definition}

A finite $p$-group $G$ is powerful if and only if $G$ embeds powerfully into itself.
If $N \leq G$ is powerfully embedded in $G$, then $[N,G] \leq N$ and therefore $N$ is a normal subgroup of $G$.
We will write $N \pe G$ if $N$ is powerfully embedded in $G$.

\begin{lemma}[Lemma 2.4 of \cite{dixon}]
	Suppose $G$ is a powerful $p$-group. Then for each $i \geq 1$:
	\begin{itemize}
		\item[(a)] $G_i$ is powerfully embedded in $G$ and $G_{i+1} = G_i^p = \Phi(G_i)$.
		\item[(b)] The map $\Theta_i : G_i/G_{i+1} \rightarrow G_{i+1}/G_{i+2}, \; x \mapsto x^p$ is a well-defined surjective homomorphism.
	\end{itemize}
\end{lemma}
\begin{proof}
	As $G$ is powerful, $G_1 = G$ is powerfully embedded into itself. Now we argue by induction. Assume $G_i \pe G$. Then, by the above Proposition, $G_i^p \pe G$.
	Notice $G_{i+1} = G_i^p[G_i,G]= G_i^p$. Hence $G_i \pe G$ for each $i \geq 1$. 
	
	We have $G_i^p \leq G_i^p[G_i,G_i] \leq G_i^p[G_i,G] = G_{i+1} = G_i^p$. This shows (a).
	
	In particular, each $G_i$ is powerful and we have $G_{i+1} = P_2(G_i)$ and $G_{i+2} = P_3(G_i)$. It is sufficient to prove (b) in the case $i = 1$ as we can replace $G$ by $G_i$.
	Additionally we can assume $G_3 = 1$ by replacing $G$ by $G/G_3$.
	
	Now $[G_2,G] \leq G_3 = 1$ and thus $G_2 \leq Z(G)$. Hence $[G,G] \leq G_2 \leq Z(G)$.
	Given $x,y \in G$, we have $(xy)^p = x^py^p [x,y]^{p(p-1)/2}$, so $p(p-1)/2$ is divided by $p$ and we obtain $[y,x]^{p(p-1)/2} \in [G,G]^p \leq G_2^p = G_3 = 1$. Thus $(xy)^p = x^p y^p$. The image of $G_2$ under $\Theta$ is contained in $G_3$ as $G_2^p = G_3$.
	Therefore $\Theta$ is a well-defined surjective homomorphism.
\end{proof}

\begin{lemma}[Lemma 2.5 of \cite{dixon}]
	Suppose $G = \langle a_1, \dots, a_d \rangle$ is a powerful $p$-group. Then $G^p = \langle a_1^p, \dots, a_d^p \rangle$.
\end{lemma}
\begin{proof}
	Let $\Theta: G/G_2 \rightarrow G_2/G_3, \; x \mapsto x^p$ be the homomorphism from the previous Lemma. As $\Theta$ is surjective, $G_2/G_3$ is generated by 
	$\{ a_1G_3, \dots a_dG_3 \}$. Therefore $G_2 = \langle a_1^p, \dots , a_d^p \rangle G_3$. Now $G_3 = \Phi(G_2)$ is the Frattini subgroup of $G_2$ and therefore
	$G^p = G_2$ is generated by $\{ a_1^p, \dots , a_d^p \}$.
\end{proof}

\begin{proposition}[Proposition 2.6 of \cite{dixon}]
	Suppose $G$ is a powerful $p$-group. Then every element of $G^p$ is a $p$-th power.
\end{proposition}
\begin{proof}
	We prove the statement by induction on $|G|$. Fix $g \in G^p$. Then $gG_3$ is contained in the image of $\Theta$. Hence there is some $x \in G$ such that $x^pG_3 = gG_3$.
	Set $y = (x^p)^{-1} g \in G_3$.
	
	Now consider the group $H = \langle G^p, y \rangle$. $H$ is powerful. As $y \in G_3 = G_2^p \subseteq G^p$ we obtain $g \in H^p$.
	If $H \neq G$ then by induction $g$ is a $p$-th power in $H$.
	If $G = H$ then $G = \langle x \rangle$ as $\Phi(G) = G^p$.
\end{proof}

Inductively we obtain the following theorem from the previous results
\begin{theorem}[Theorem 2.7 of \cite{dixon}] \label{thm:finite_powerful}
	Suppose $G = \langle a_1, \dots, a_d \rangle$ is a powerful $p$-group. Then the following hold:
	\begin{itemize}
		\item[(a)] $P_{k+1}(G_i) = G_i^{p^k}$ for all $i \geq 1$ and $k \geq 0$.
		\item[(b)] $G_i = G^{p^{(i-1)}} = \{ x^{p^{(i-1)}} : x \in G \} = \langle a_1^{p^{(i-1)}}, \dots, a_d^{p^{(i-1)}} \rangle$ for all $i \geq 1$.
		\item[(c)] The map $x \mapsto x^{p^k}$ induces a surjective homomorphism $G_i/G_{i+1} \rightarrow G_{i+k}/G_{i+k+1}$ for all $i \geq 1$ and $k \geq 0$.
	\end{itemize}
\end{theorem}

We now see that a finitely generated powerful $p$-group is a product of cyclic groups:

\begin{corollary}[Corollary 2.8 of \cite{dixon}] \label{cor:finite_powerful_product}
	Suppose $G = \langle a_1, \dots, a_d \rangle$ is a powerful $p$-group. Then $G = \langle a_1 \rangle \cdots \langle a_d \rangle$ is 
	a product of cyclic groups.
\end{corollary}
\begin{proof}
	Let $l$ be the length of the lower $p$ series, i.e. $G_l > G_{l+1} = 1$. By induction on $l$, we may assume that the claim holds for $G/G_l$.
	Hence
	\[ G/G_l = \langle a_1G_l , \dots, a_dG_l \rangle = \langle a_1G_l \rangle \cdots \langle a_dG_l \rangle . \]
	Therefore $G = \langle a_1 \rangle \cdots \langle a_d \rangle G_l$. But $G_l = \langle a_1^{p^{l-1}} , \dots, a_d^{p^{l-1}} \rangle$ and as 
	$G_l \pe G$, we have $[G_l,G] \leq G_l^p = G_{l+1} = 1$. Hence $G_l$ is central in $G$ and therefore 
	$G_l \subseteq \langle a_1 \rangle \cdots \langle a_d \rangle$.
\end{proof}

\subsection{Powerful pro-$p$ groups}
We follow Section 3.1 of \cite{dixon} to develop the theory of powerful pro-$p$ groups. Again we restrict to the assumption $p\neq 2$ and refer to \cite{dixon} for the general case.

\begin{definition}
Suppose $G$ is a pro-$p$ group.
\begin{itemize}
	\item[(a)] $G$ is \textit{powerful} if $G/\overline{G^p}$ is abelian.
	\item[(b)] An open subgroup $N \leq_o G$ is powerfully embedded in $G$ if  $[N,G] \leq \overline{N^p}$.
\end{itemize}
\end{definition}

As in the finite case we see that a  pro-$p$ group $G$ is powerful if and only if $G$ embeds powerfully into itself.
Furthermore, if $N \leq_o G$ is powerfully embedded in $G$, then $N$ is a normal subgroup of $G$, so
we continue to write $N \pe G$ if $N$ is powerfully embedded in $G$.
The exact connection between powerful $p$-groups and powerful pro-$p$ groups is given  by

\begin{proposition}[Corollary 3.3 of \cite{dixon}]
	A topological group $G$ is powerful pro-$p$ if and only if it is the inverse limit of a surjective inverse system of powerful $p$-groups.
\end{proposition}
Using this connection it is not hard to see that the following results carry over from the finite to the profinite situation:

\begin{theorem}[Theorem 3.6 of \cite{dixon}] \label{powerful_thm}\label{powerful_product}
	Suppose $G = \overline{\langle a_1 , \dots, a_d \rangle }$ is a finitely generated powerful pro-$p$ group and $i \geq 0$.
	Then the following hold:
	\begin{itemize}
		\item[(a)] $G_i \pe G$.
		\item[(b)] $G_{i+k} = P_{k+1}(G_i) = G_i^{p^k}$ for all $k \geq 0$. In particular $\Phi(G_i) = G_{i+1}$.
		\item[(c)] $G_i = G^{p^{i-1}} = \{ x^{p^{i-1}} : x \in G \} = \overline{ \langle a_1, \dots , a_d \rangle }$.
		\item[(d)] The map $x \mapsto x^{p^k}$ induces a surjective continuous homomorphism $G_i/G_{i+1} \rightarrow G_{i+k}/G_{i+k+1}$ for every $k \geq 0$. 
		\item[(e)] $G = \overline{ \langle a_1 \rangle} \cdots \overline{ \langle a_d \rangle}$.
	\end{itemize}
\end{theorem}

\subsection{Powerful pro-$p$ groups and finite rank}

If $G$ is a powerful $p$-group and $H\leq G$, then $d(H) \leq d(G)$ (cp Theorem 2.9 of \cite{dixon}). This transfers to closed subgroups of profinite groups by  
Proposition~\ref{prop:gen_set} 
and yields in particular that  every finitely generated powerful pro-$p$ group has (indeed) finite rank:

\begin{theorem}[Theorem 3.8 of \cite{dixon}]\label{t:finite_rank}
	Suppose $G$ is a finitely generated powerful pro-$p$ group and $H \leq_c G$ is a closed subgroup. Then $d(H) \leq d(G)$. In particular, $G$ has finite rank.
\end{theorem} 

We next show a kind of converse, namely every pro-$p$ group of finite rank has a characteristic open powerful subgroup. Here is the candidate:

\begin{definition}
	Given a finite (or profinite) $p$-group $G$, we define
	\[ V(G,r) = \bigcap \{\ \ker f \ |\ f:G \rightarrow GL_r(\mathbb{F}_p) \text{ is a homomorphism } \}. \]
\end{definition}

Note that if $G$ is a finitely generated pro-$p$ group and $r$ is a positive integer, then $V(G,r)$ is an open characteristic subgroup of $G$ because every homomorphism $f:G \rightarrow GL_r(\F_p)$ is continuous by Theorem~\ref{thm:serre} and there are only finitely many such homomorphisms by Proposition~\ref{prop:fin_gen_fin_subgps}.

The nice thing about $V(G,r)$ is the fact that for any finite $p$-group $G$ any normal subgroup $N$ with  $\gpdim(N) \leq r$ and $N \leq V$, is powerfully embedded in $V$ (cp. Proposition 2.12 of \cite{dixon}). This transfers to the profinite setting as:

\begin{proposition}[Proposition 3.9 of \cite{dixon}]\label{powerful_char}
	Suppose $G$ is a finitely generated pro-$p$ group, $r$ is a positive integer and $N \triangleleft_o G$ is an open normal subgroup with $\gpdim(N) \leq r$.
	Set $V = V(G,r)$. If  $N \leq V$, then $N \pe V$. In particular, $V$ itself is powerful.
\end{proposition}

\subsection{Uniform pro-$p$ groups}
In order to further continue the analogy  between a powerful and an abelian pro-$p$ group we would like to have that the elementary abelian $p$-groups arising as quotients of the lower $p$-series all have the same $\F_p$-dimension. We follow Section 4.1 of \cite{dixon}.
 
\begin{definition}
	A pro-$p$ group $G$ is \textit{uniformly powerful} (or \textit{uniform}) if
	\begin{itemize}
		\item[(a)] $G$ is finitely generated,
		\item[(b)] $G$ is powerful, and
		\item[(c)] for each $i\geq 1$, the map $x \mapsto x^p$ induces an isomorphism
			$G_i/G_{i+1} \rightarrow G_{i+1}/G_{i+2}$.
	\end{itemize}
\end{definition}

\begin{theorem}[Theorem 4.2 of \cite{dixon}]
	Let $G$ be a finitely generated powerful pro-$p$ group. Then $P_k(G)$ is uniform for all sufficiently large $k$.
\end{theorem}
\begin{proof}
	Write $p^{d_i} = | G_i : G_{i+1} |$ for all $i \geq 1$. As the map $x \mapsto x^p$ induces an epimorphism
	$G_i / G_{i+1} \twoheadrightarrow G_{i+1} / G_{i+2}$, we have $d_1 \geq d_2 \geq d_3 \geq \dots$ and hence there is some $m$ such that $d_i = d_m$ for all $i \geq m$. Then the group $G_k$ is uniform by Theorem~\ref{powerful_thm}.
\end{proof}

\begin{corollary}[Corollary 4.3 of \cite{dixon}]
	Any pro-$p$ group $G$  of finite rank  has a uniform characteristic open subgroup.
\end{corollary}
\begin{proof}
	$G$ has a powerful characteristic open subgroup $H$ by Proposition~\ref{powerful_char}. Then, by the above theorem, there is $k \geq 1$ such that $H_k$ is uniform. $H_k$ is an open characteristic subgroup of $H$ and hence an open characteristic subgroup of $G$.
\end{proof}
In fact, one can now see that a pro-$p$ group $G$ has finite rank if and only if it has an open uniform subgroup $N$ since 	$\rk(G) \leq \rk(N) + \rk(G/N).$

\subsection{$\mathbb{Z}_p$-action and good bases}
By Theorem~\ref{powerful_product} every uniform pro-$p$ group $G$ is a product of procyclic subgroups. We show that this gives rise to a homeomorphism $\Z_p^d \rightarrow G$ where $d = \gpdim(G)$. We use this to define \emph{good bases} for open subgroups of $G$. These are special generating sets which were introduced by du Sautoy in \cite{sautoy} and allow a kind of 'stratification' of the underlying group. They will be used in Section 5 to prove that the open subgroups of a uniform pro-$p$ group are uniformly definable in $\Z_p^\text{an}$.

We start with defining a $\Z_p$-action on a pro-$p$ group $G$
\begin{definition}
	Let $G$ be a pro-$p$ group, $g \in G$ and $\lambda \in \Z_p$. Then $g^\lambda$ is the limit of $(g^{a_i})$ in $G$, where $(a_i)$ is a sequence of integers that converges to $\lambda$ in $\Z_p$.
\end{definition}

It is not hard to see that this is well-defined, i.e. if $(a_i),(b_i)$ are two sequences of integers that have the same limit in $\Z_p$, then for any $g \in G$ the sequences $(g^{a_i})$ and $(g^{b_i})$ converge and have the same limit	in $G$. This definition extends in a natural way to an action of the ring $\Z_p$ and  turns $G$ into a (topological) $\Z_p$-module:

\begin{proposition}[Proposition 1.26 of \cite{dixon}]
	Suppose $G$ is a pro-$p$ group, $g,h \in G$ and $\lambda, \mu \in \Z_p$.
	\begin{itemize}
		\item[(a)] $g^{\lambda + \mu} = g^\lambda g^\mu$ and $g^{\lambda \mu} = (g^\lambda)^\mu$.
		\item[(b)] If $gh = hg$ then $(gh)^\lambda = g^\lambda h^\lambda$.
		\item[(c)] The map $\nu \mapsto g^\nu$ is a continuous homomorphism from $\Z_p$ to $G$ whose image is the closure of $\langle g \rangle$ in $G$.
	\end{itemize}
\end{proposition}
\begin{proof}
	Note that (a) and (b) hold in $G/N$ for all open normal subgroups $N \trianglelefteq_o G$. Hence they hold in $G$.	
	(a) implies that $\nu \mapsto g^\nu$ is a homomorphism. Every homomorphism with domain $\Z_p$ is continuous.
	$g^{\Z_p}$ is the image of a compact group, hence compact and thus closed. Clearly $\langle g \rangle \subseteq g^{\Z_p}$.
	Each element of $g^{Z_p}$ is the limit of a sequence in $\langle g \rangle$ and therefore $g^{\Z_p} \subseteq \overline{\langle g \rangle}$.
	Hence $g^{\Z_p} = \overline{\langle g \rangle}$. 
\end{proof}

Let $G$ be a uniform pro-$p$ group. Recall that the map $x \mapsto x^p$ induces an isomorphism
\[ f_i : G_i / G_{i+1} \rightarrow G_{i+1} / G_{i+2} \]
for all $i$. Moreover, by Theorem~\ref{powerful_thm} we have $G_{i+1} = \Phi(G_i)$ for all $i$ and hence the quotient
$G_i/G_{i+1}$ is a $\gpdim(G)$ dimensional $\F_p$-vector space.

\begin{proposition}[Theorem 4.9 of \cite{dixon}] \label{free_module}
	Suppose $G$ is a uniform pro-$p$ group with $d = \gpdim(G)$ and $\{a_1, \dots, a_d \}$ is a generating set for $G$. Then the map
	\[ \psi: \Z_p^d \rightarrow G, (\lambda_1, \dots, \lambda_d) \mapsto a_1^{\lambda_1} \cdots a_d^{\lambda_d} \]
	is a homeomorphism.
\end{proposition}
\begin{proof}
	Fix $a \in G$. By Theorem~\ref{powerful_product}, $G = \overline{ \langle a_1 \rangle} \cdots \overline{ \langle a_d \rangle}$ so, by the above proposition, we can find $\lambda_1, \dots, \lambda_d \in \Z_p$
	such that $a = a_1^{\lambda_1} \cdots a_d^{\lambda_d}$.
	It remains to show that these $\lambda_i$ are unique. Fix $k>0$ and consider the finite powerful group $G/G_{k+1}$.
	$G_{k+1}$ has index $p^{kd}$ in $G$ and by  Corollary~\ref{cor:finite_powerful_product}, we have
	\[ G/G_{k+1} = \langle a_1 G_{k+1} \rangle \cdots \langle a_d G_{k+1} \rangle. \]
	By Theorem~\ref{thm:finite_powerful}, each cyclic subgroup of $G/G_{k+1}$ has order at most $p^k$. Hence the groups $\langle a_i G_{k+1} \rangle$ have index exactly $p^k$ in $G/G_{k+1}$.
	We have $a_i^{p^k} \equiv 1 \pmod{G_{k+1}}$ and hence every element $bG_{k+1}$ can be written as
	\[ bG_{k+1} = a_1^{e_1} \cdots a_d^{e_d} G_{k+1} \text{ for some } e_i \in \Z/p^k\Z. \]
	But as $|G:G_{k+1}| = p^{kd}$, these $e_i$ must be uniquely determined.
	
	Hence the $\lambda_i$ are uniquely determined modulo $p^k$ for all $k$ and thus are uniquely determined.
\end{proof}

If $G$ is a uniform pro-$p$ group with $d = \gpdim(G)$ and $\lambda = (\lambda_1, \dots \lambda_d) \in \Z_p^d$, we denote the corresponding element of $G$ by $x(\lambda) \in G$. If $x \in G$ is an element of $G$, we denote the corresponding tuple $\lambda = (\lambda_1, \dots \lambda_d) \in \Z_p^d$ by $\lambda(x)$.

\begin{definition}
	Let $G$ be a pro-$p$ group. Then we define $\omega: G \rightarrow \N \cup \{\infty \}$ by $\omega(g) = n$ if $g \in P_n(G) \setminus P_{n+1}(G)$ and put $\omega(1) = \infty$.
\end{definition}
Note that $\omega(ab) \geq \min\{\omega(a), \omega(b)\}$ and $\omega(ab) = \omega(a)$ if $\omega(a) < \omega(b)$.

Let $\nu$ be the usual valuation on $\Z_p$. If $G$ is a uniform pro-$p$ group then $\omega$ is compatible with $\nu$.

\begin{lemma}
	Let $G$ be a uniform pro-$p$ group, $g \in G$, and $\lambda \in \Z_p$. Then
	\[ \omega(g^\lambda) = \omega(g) + \nu(\lambda). \]
\end{lemma}
\begin{proof}
	Set $d = \gpdim(G)$.
	In case $g = 1$ or $\lambda = 0$ there is nothing to show. Suppose $g \neq 1$ and $\lambda \neq 0$.
	Set $k = \nu(\lambda)$. Then $\lambda = rp^k \rho$ with $0 \leq r < p$ and $\nu(\rho) > k$.	
	As $g^{p^{k+1}} \in G_{\omega(g)+k+1}$ and $\nu(\rho) > k$, it follows $g^\rho \in  G_{\omega(g)+k+1}$.
	Hence $g^\lambda \equiv g^{rp^k} \pmod{ G_{\omega(g)+k+1}}$ and therefore $g^\lambda \in G_{\omega(g)+k+1}$.
	
	The map $x \mapsto x^{p^k}$ induces an isomorphism
	\[ f: G_{\omega(g)}/G_{\omega(g)+1} \rightarrow G_{\omega(g)+k}/G_{\omega(g)+k+1}. \]
	Hence $g^{p^k} \in G_{\omega(g)+k} \setminus G_{\omega(g)+k+1}$ and hence $g^{rp^k} \in G_{\omega(g)+k} \setminus G_{\omega(g)+k+1}$ because $r$ does not divide
	$p^d = | G_{\omega(g)+k} : G_{\omega(g)+k+1} |$.
\end{proof}

\begin{proposition}[Theorem 1.18 (iv) of \cite{sautoy}]  \label{omega_x_lambda}
	Let $G$ be a uniform pro-$p$ group, $d = \gpdim(G)$, and $\{ x_1, \dots, x_d \}$ a generating set for $G$.
	If $x = x(\lambda)$ then $\omega(x) = \min \{ \nu(\lambda_i) + 1 : i = 1, \dots, d \}$.
\end{proposition}
\begin{proof}
	As $G_2 = \Phi(G)$ is the set of nongenerators, we have $\omega(x_i) = 1$ for each $i$. In particular, $\omega(x_i^{\lambda_i}) \geq n :=  \min \{ \nu(\lambda_i) + 1 : i = 1, \dots, d \}$.
	Therefore $x_i^{\lambda_i} \in G_n$ for each $i$. Write $x_i^{\lambda_i} \equiv x_i^{r_ip^{n-1}} \pmod{G_n}$. By the minimality of $n$, not all $r_i$ are zero.
	Note that the $x_i^{p^{n-1}}$ are a basis of the vector space $G_n/G_{n+1}$. We have
	\[ x(\lambda) \equiv x_1^{r_1p^{n-1}} \cdot \dots \cdot x_d^{r_dp^{n-1}} \pmod{G_{n+1}}. \]
	This is a nontrivial linear combination of base vectors and thus nontrivial.
	Thus $x(\lambda) \not \in G_{n+1}$.		
\end{proof}

We will use a construction from Section 2 of \cite{sautoy}.
Let $G$ be a uniform pro-$p$ group. If $x(\lambda) \in G_n$. Then $n \leq \omega(x(\lambda)) = \min \{ \nu(\lambda_i) + 1 : i = 1, \dots, d \}$ and hence
$\nu(\lambda(i)) \geq n-1$ for all $i = 1, \dots d$. Therefore $p^{-(n-1)} \lambda_i \in \Z_p$ for all $i$.
We define the map 
\[ \pi_n : G_n \rightarrow \F_p^d, x(\lambda) \mapsto ( \pi(p^{-(n-1)} \lambda_1), \dots \pi(p^{-(n-1)} \lambda_d) ) \]
where $\pi: \Z_p \rightarrow \F_p$ is the residue map.
If $x(\lambda^1 = (\lambda_1^1, \dots \lambda_d^1) )$ and $x(\lambda^2 = (\lambda_1^2, \dots \lambda_d^2))$ are elements of $G_n$ then $x_i^{\lambda_i^k} \in G_n$ for $i = 1, \dots d$ and $k = 1,2$.
Take $\lambda \in \Z_p^d$ such that $x(\lambda) = x(\lambda_1) x(\lambda_2)$. Then
\[ x_1^{\lambda_1} \cdots x_d^{\lambda_d} \equiv x_1^{\lambda_1^1+ \lambda_1^2} \cdots  x_d^{\lambda_d^1+ \lambda_d^2} \pmod{G_{n+1}} \] 
as $G_n/G_{n+1}$ is abelian.
Hence $x(\lambda^1 + \lambda^2 - \lambda ) \equiv 1 \pmod{G_{n+1}}$ and thus $x(\lambda^1 + \lambda^2 - \lambda ) \in G_{n+1}$.
This implies
\[ n+1 \leq \omega(x(\lambda^1 + \lambda^2 - \lambda )) = \min \{ \nu(\lambda_i^1+ \lambda_i^2-\lambda_i) +1 : i = 1, \dots d \}\]
and therefore $\nu(\lambda_i^1+ \lambda_i^2-\lambda_i) \geq n$ for all $i$. We see that $\pi_n(x(\lambda)) = \pi_n(x(\lambda^1+\lambda^2))$ and therefore $\pi_n$ is a homomorphism.

Moreover, $x(\lambda) \in \ker \pi_n$ if and only if $\nu(p^{-(n-1)} \lambda_i) > 0$ for all $i$. But this is equivalent to
$\omega(x(\lambda)) > n$ and hence $\ker \pi_n = G_{n+1}$. Therefore $\pi_n$ is an isomorphism between the $d$ dimensional
$\F_p$ vector spaces $G_n/G_{n+1}$ and $\F_p^d$.

Let $f_i : G_i/G_{i+1} \rightarrow G_{i+1}/G_{i+2}$ be the isomorphism induced by $x \mapsto x^p$.
Then 
\[ f_i(x(\lambda) G_{i+1} ) = f_i(x_1^{\lambda_1}G_{i+1}) \cdots f_i(x_d^{\lambda_d}G_{i+1}) 
		= x_1^{p \lambda_1}G_{i+2} \cdots x_d^{p \lambda_d}G_{i+2} = x(p\lambda) G_{i+2}. \]

\begin{definition}
	Let $G$ be a uniform pro-$p$ group, $d = \gpdim(G)$, and let $H \leq_o G$ be an open subgroup.
	A tuple $(h_1, \dots , h_d) \in H^d$ is a \textit{good basis} of $H$ if
	\begin{itemize}
		\item[(a)] $\omega(h_i) \leq \omega(h_j)$ whenever $i \leq j$, and
		\item[(b)] for each $n$, $\{ \pi_n(h_j) : j \in I_n \}$ extends the linearly independent set $\{ \pi_n(h_j^{p^{n-\omega(h_j)}}) : j \in I_1 \cup \dots \cup I_{n-1} \}$
			to a basis of $\pi_n(H \cap G_n)$, where $I_n := \{ j : \omega(h_j) = n \}$.
	\end{itemize} 
\end{definition}

\begin{lemma} \label{good_base_exist}
	Let $G$ be a uniform pro-$p$ group and let $H \leq_o G$ be an open subgroup. Then there is a good basis for $H$.
\end{lemma}
\begin{proof}
	Set $d = \gpdim(G)$. Assume $(h_1, \dots, h_m)$ satisfy (a) and (b) up to $r$ for some $r \geq 0$.
	Note that $G_N \leq H$ for sufficiently large $N$ and therefore
	\[ \pi_N(H \cap G_N) = \pi_N(G_N) = \F_p^d. \]
	If $m < d$ then there is a minimal $s> r$ such that $\dim_{\F_p}(\pi_s(H \cap G_s) ) > \dim_{\F_p}(\pi_r(H \cap G_r ))$.
	As $\pi_s : G_s/G_{s+1} \rightarrow \F_p^d$ is an isomorphism, we can find $h_{m+1}, \dots, h_{m+k} \in H \cap G_s \setminus G_{s+1}$ such that
	$\{ \pi_s(h_{m+1}) , \dots , \pi_s(h_{m+k}) \}$ extends the linearly independent set $\{ \pi_s( h_j^{p^{n-\omega(h_j)}}) : j \leq m \}$ to a basis of $\pi_s( H \cap G_s )$. 
\end{proof}

If $(h_1, \dots, h_d)$ is a good basis and  $\lambda_1, \dots, \lambda_d \in \Z_p$ are elements of $\Z_p$ then we will denote
$h_1^{\lambda_1} \cdots h_d^{\lambda_d}$ by $h(\lambda)$. The valuation-like map $\omega$ and the $p$-adic valuation $\nu$ are compatible with respect to good bases.  Using the fact that the quotients of the lower $p$-series are $\F_p$-vector spaces one then shows that a \emph{good basis} deserves its name in the following sense:
\begin{lemma}
Suppose $H \leq_o G$  is an open subgroup and $(h_1, \dots, h_d)$  a good basis for $H$. Then for every $h \in H$ there are (unique) $\lambda_1, \dots, \lambda_d \in \Z_p$ such that $h = h_1^{\lambda_1} \cdots h_d^{\lambda_d}$.
	Furthermore, if $h = h(\lambda)$, then $\omega(h) = \min \{ \omega(h_i) + \nu(\lambda_i) : i = 1, \dots, d \}$.
\end{lemma}	

With this one can finally prove du Sautoy's characterization of good bases.

\begin{lemma}[Lemma 2.5 of \cite{sautoy}] \label{good_basis_lemma}
	Let $G$ be a uniform pro-$p$ group, $d = \gpdim(G)$, and $(h_1, \dots h_d ) \in G^d$. Then $(h_1, \dots h_d)$ is a good basis
	for some open subgroup of $G$ if and only if
	\begin{itemize}
		\item[(a)] $\omega(h_i) \leq \omega(h_j)$ whenever $i \leq j$;
		\item[(b)] $h_i \neq 1$ for $i = 1, \dots d$;
		\item[(c)] the set $\{ h_1^{\lambda_1} \cdots h_d^{\lambda_d} : \lambda_i \in \Z_p\}$ is a subgroup of $G$; and
		\item[(d)] for all $\lambda_1, \dots \lambda_d \in \Z_p$, $\omega(h(\lambda)) = \min\{ \omega(h_i) + \nu(\lambda_i) : i = 1, \dots d \}$.
	\end{itemize}
\end{lemma}
\begin{proof}
	Let $(h_1, \dots h_d)$ be a good basis for some open subgroup $H$. Then (a) holds by definition and (c) and (d) hold by the
	previous lemma. Note that $\pi_n(H \cap G_n ) = \pi_n(G_n) = \F_p^d$ for sufficiently large $n$. Then
	$\{ \pi_n( h_i^{p^{e(n,i)}} ) : i = 1, \dots d \}$ is a basis for $\F_p^d$. In particular, $h_i \neq 1$ for all $i$.
	
	Now assume $(h_1, \dots h_d)$ satisfies (a) to (d). Set $H = \{ h_1^{\lambda_1} \cdots h_d^{\lambda_d} : \lambda_1 , \dots \lambda_d \in \Z_p \}$. Then $H$ is a closed subgroup of $G$.
	We have $H \cap G_k = \{ h_1^{\lambda_1} \cdots h_d^{\lambda_d} : \omega(h_1^{\lambda_1} \cdots h_d^{\lambda_d}) \leq k \}$ and hence (d) implies
	\[ H \cap G_k = \{ h_1^{\lambda_1} \cdots h_d^{\lambda_d} : \lambda_i \in p^{e(k,i)} \Z_p \}. \]
	Therefore $\{ h_i^{p^{e(k,i)}} : i \in I_1 \cup \dots \cup I_k\}$ generates the $\F_p$ vector space $H\cap G_k / H \cap G_{k+1}$.
	As $\pi_n : G_n / G_{n+1} \rightarrow \F_p^d$ is an isomorphism of vector spaces, the set $\{ \pi_k( h_i^{p^{e(k,i)}}) : i \in I_1 \cup \dots \cup I_k \}$ generates $\pi_k(H \cap G_{k+1})$.
	Set $m = | I_1 \cup \dots \cup I_k|$. Note that the  $\pi_i^{p^{e(k,i)}}, i = 1, \dots m$ are linearly independent if and only if there are no $\lambda_1, \dots \lambda_m \in \Z_p \setminus p \Z_p$ such that
	$\pi_k( h_1^{\lambda_1 p^{e(k,1)}} \cdots h_d^{\lambda_d p^{e(k,d)}} ) = 0$, or equivalently
	$\omega(h_1^{\lambda_1 p^{e(k,1)}} \cdots h_d^{\lambda_d p^{e(k,d)}}) > k$.
	But $\nu(\lambda_i) = 0$ for all $\lambda_i \in \Z_p \setminus p \Z_p$ and hence
	\[ \omega(h_1^{\lambda_1 p^{e(k,1)}} \cdots h_d^{\lambda_d p^{e(k,d)}}) = \min\{ \omega(h_i) + \nu(\lambda_i) + e(k,i) \} = k.\]
	It remains to show that $H$ is open. By (b), $\pi_n(H  \cap G_n)$ has dimension $d$ for sufficiently large $n$.
	Then $\pi_n(H \cap G_k) = \F_p^d = \pi_n(G_n)$ and hence $(H \cap G_n)G_{n+1} = G_n$. But then $H \cap G_n = G_n$ as $G_{n+1}$ is the set of nongenerators of $G_n$. Therefore, $H$ contains the open subgroup $G_n$ and is thus open.
\end{proof}

\section{Compact $p$-adic analytic groups}
\subsection{$p$-adic analytic groups}
In this section we will explain the notions of $p$-adic analytic manifolds and $p$-adic analytic groups.
We follow Section 8 in \cite{dixon}.

Let $r \geq 0$ be a positive integer. A basis for the topology of $\Z_p^r$ is given by the sets
\[B(y,p^{-h}) = \{ z \in \Z_p^r : |z_i - y_i| \leq p^{-h} \text{ for } i = 1, \dots, r \} = \{ y+p^hx : x \in \Z_p^r \}\]
where $y \in \Z_p^r$, $h \geq 0$, and $| \cdot |$ is the usual norm on $\Z_p$. We are interested in functions that can be described in terms of formal power series over $\Q_p$. 

\begin{definition}
	Let $V \subseteq \Z_p^r$ be a nonempty open subset and $f : V \rightarrow \Z_p^s$ be a function with components $f = (f_1, \dots, f_s)$.
	\begin{itemize}
		\item[(a)] The function $f$ is analytic at $y \in V$ if there is $h \in \N$ such that $B(y,p^{-h}) \subseteq V$ and there are formal power series $F_i(X) \in \Q_p[[X]]$ such that
			\[f_i(y+p^hx) = F_i(x) \text{ for all } x \in \Z_p^r. \]
		\item[(b)] The function $f$ is analytic on $V$ if it is analytic at each point of $V$.
	\end{itemize}
\end{definition}

\begin{definition}
	Let $X$ be a topological space.
	\begin{itemize}
		\item[(a)] A \textit{chart} on $X$ is a tuple $(U,\phi,n)$ where $U \subseteq X$ is a non-empty open subset and $\phi$ is a homeomorphism from $U$ onto an open subset of $\Z_p^n$.
		\item[(b)] Two charts $(U,\phi,n)$ and $(V, \psi,m)$ are \textit{compatible} if both
			$\psi \circ \phi^{-1}$ and $\phi \circ \psi^{-1}$ are analytic functions on $\phi(U \cap V)$ respectively $\psi(U \cap V)$. 
		\item[(c)] An \textit{atlas} on $X$ is a family $\{ (U_i, \phi_i, n_i) : i \in I \}$ of pairwise compatible charts such that
			$X = \bigcup_{i \in I} U_i$.
	\end{itemize}
\end{definition}

As usual, two atlases $A$ and $B$ are compatible if every chart in $A$ is compatible with every chart in $B$. This is an equivalence relation on the set of atlases on $X$. Hence we can give the definition of a $p$-adic analytic manifold.

\begin{definition}
	\begin{itemize}
		\item[(a)] A $p$-adic analytic manifold is a topological space $X$ together with an equivalence class of compatible atlases on $X$.
		\item[(b)] A function $f : X \rightarrow Y$ between two $p$-adic analytic manifolds is \textit{analytic} if for each pair of charts $(U, \phi, n)$ of $X$ and $(V, \psi, m)$ of $Y$ the following hold:
		\begin{itemize}
			\item[(i)] $f^{-1}(V)$ is open, and
			\item[(ii)] $\psi \circ f \circ \phi^{-1}$ is analytic on $\phi(U \cap f^{-1}(V) )$.
		\end{itemize}
	\end{itemize}
\end{definition}

Note that 
any $f : X \rightarrow Y$  analytic function between analytic manifolds is continuous (see e.g. Lemma 8.13 of \cite{dixon}).

If $X$ and $Y$ are $p$-adic analytic manifolds then the space $X \times Y$ has naturally the structure of a $p$-adic analytic manifold (see \cite[Examples 8.9 (vi)]{dixon}). A $p$-adic analytic group is a $p$-adic analytic manifold equipped with analytic group operations.

\begin{definition}
	A $p$-adic analytic group $G$ is a topological group that is a $p$-adic analytic manifold such that both
	\[ (x,y) \mapsto x \cdot y \text{ and } x \mapsto x^{-1} \]
	are analytic functions.
\end{definition}

\subsection{Analytic structure on uniform groups}
By Lazard's theorem a topological group has the structure of a $p$-adic analytic group if and only if it has an open uniformly powerful pro-$p$ subgroup.

Let $G$ be a uniform pro-$p$ group, $d = \gpdim(G)$, and fix a generating set $\{x_1, \dots x_d \}$ for $G$.
By Proposition~\ref{free_module} the map
\[ \Z_p^d \rightarrow G, (\lambda_1, \dots \lambda_d) \mapsto x_1^{\lambda_1} \cdots x_d^{\lambda_d} \]
is a continuous bijection between $\Z_p^d$ and $G$ and hence a homeomorphism as both $G$ and $\Z_p^d$ are compact Hausdorff spaces.

This homeomorphism gives us a global chart on $G$ and hence we may view $G$ as a compact $p$-adic analytic manifold of dimension $d$. We cite two important facts.

\begin{theorem}[Theorem 8.18 of \cite{dixon}]
	$G$ is a compact $p$-adic analytic group with respect to the $p$-adic analytic manifold structure induced by the above homeomorphism $\Z_p^d \cong G$.
\end{theorem}

Moreover, continuous homomorphisms and analytic homomorphisms between $p$-adic analytic groups coincide by the following theorem.

\begin{theorem}[Theorem 9.4 of \cite{dixon}] \label{cont_analytic}
	Every continuous homomorphism between $p$-adic analytic groups is analytic.
\end{theorem}

\section{ NIP and NTP\textsubscript{2}}
We now turn to the model theoretic side and recall some definitions around the independence property and the tree property.
We consider formulas $\phi(x,y)$ where $x ,y$ are tuples of variables.

\begin{definition}
	Suppose $\mathcal{L}$ is a language, $T$ is a complete $\mathcal{L}$-theory and $\phi(x, y)$ is an $\mathcal{L}$-formula.
	The formula $\phi(x, y)$ has the \textit{independence property} (IP) if there is a model
	$M \models \mathcal{T}$ and constants $(a_i: i<\omega)$, $(b_A : A \subseteq \mathbb{N} )$ such that
	$M \models \phi(a_i,b_A) \iff i \in A$.		
	The theory $T$ has NIP if no formula in $T$ has IP.
\end{definition}

We say that a model $M$ has NIP if its theory $\Th(M)$ has NIP.
By the compactness theorem, a formula $\phi(x,y)$ has IP if and only if for all finite sets $F \subseteq \N$ there are $(a_i : i \in F)$ and $(b_A : A \subseteq F)$ such that
\[ \forall i \in F, A \subseteq F : M \models \phi(a_i,b_A) \iff i \in A. \]
If $F$ is fixed, this is a finite condition and does not depend on the model.

Furthermore also by compactness,  a formula  $\phi(x,y)$  has NIP if and only if the formula $\phi^{\text{opp}}(y,x) = \phi(x,y)$ has NIP.

The model theoretic notion of NIP is closely connected to the combinatorial Vapnik-Chervonenkis dimension, see e.g. Section 6.1 in \cite{simon}.

\begin{definition}
	Let $X$ be a set and let $\mathcal{S}$ be a family of subsets of $X$.
	\begin{itemize}
		\item[(a)] A subset $A \subseteq X$ is \emph{shattered} by $\mathcal{S}$ if for every $A' \subseteq A$ there is $S \in \mathcal{S}$ such that $A' = S \cap A$.
		\item[(b)] We say $\mathcal{S}$ has \emph{Vapnik-Chervonenkis dimension} $n$, $\VC(\mathcal{S}) = n$, if it shatters a subset of size $n$ and does not shatter any subset of size $n+1$. If $\mathcal{S}$ shatters a subset of size $n$ for all $n$, then we say that $\mathcal{S}$ has infinite Vapnik-Chervonenkis dimension.
	\end{itemize}
\end{definition}

If $M$ is a model and $\phi(x,y)$ is a formula, we consider the set $X = M$ together with the family $\mathcal{S}_\phi = \{ \phi(M,b) : b \in M\}$. Then $\phi(x,y)$ has NIP if and only if $\mathcal{S}_\phi$ has finite VC-dimension. The VC-dimension of $\mathcal{S}_\phi$ does not depend on the choice of the model $M$ and we will denote it by $\VC(\phi(x,y))$.

The Baldwin-Saxl lemma states an important property for families of uniformly definable subgroup (see, for example, \cite[Theorem 2.13]{simon}).

\begin{lemma}[Baldwin-Saxl lemma] \label{baldwin_saxl}
	 Suppose $G$ is a $\emptyset$-definable group and $\phi(x,y)$ has NIP. Put $k = \VC(\phi^\mathrm{opp}(y,x))$.
	Suppose that $(H_i : i \in I)$ is a family of subgroups $H \leq G$ such that each $H_i$ can be defined by some instance of $\phi$.
	Then for all finite subsets $I_f \subseteq I$, there are $i_1, \dots, i_k \in I_f$ such that
	\[ \bigcap_{i \in I_f} H_i = H_{i_1} \cap \dots \cap H_{i_k}. \]
\end{lemma}
\begin{proof}
	Suppose the lemma fails. Then there is some subset $I_f = \{i_0, \dots , i_k \} \subseteq I$ of size $k+1$ 
	such that $\bigcap_{i \in I_f} H_i \subsetneq \bigcap_{i \in I_f \setminus \{i_n\}} H_i$
	for all $n \in \{ 0, \dots, k \}$. Fix some $a_n \in \bigcap_{i \in I_f \setminus \{i_n\}} H_i \setminus \bigcap_{i \in I_f} H_i$ for each $n$. Note that $a_n \in H_{i_t} \iff n \neq t$.
	For $F \subseteq \{0, \dots k\}$ put
	\[ a_F = \prod_{n \in \{0, \dots k\} \setminus F} a_n. \]
	Then $a_F \in H_{i_t} \iff t \in F$ and hence $\phi^\text{opp}(y,x)$ has VC dimension at least $k+1$. This contradicts the assumption.
\end{proof}

If $X$ is a set and $\mathcal{S}$ is a family of subsets, the \emph{shatter function} is given by
\[ \pi_\mathcal{S}(n) = \max\{ |\{S \cap A : S \in \mathcal{S} \} | : A \subseteq X, |A| \leq n \}. \]
We have $\pi_\mathcal{S}(n) = 2^n$ if and only if there is a set $A \subseteq X$ of size at most $n$ that is shattered by $\mathcal{S}$. Hence $VC(\mathcal{S}) = n$ if and only if $n$ is maximal such that $\pi_\mathcal{S}(n) = 2^n$ and $\mathcal{S}$ has infinite VC dimension if and only if $\pi_\mathcal{S}(n) = 2^n$ for all $n$.
By the Sauer-Shelah lemma (see, for example, \cite[Lemma 6.4]{simon}), either $\pi_\mathcal{S}(n) = 2^n$ for all $n$ or the shatter function is bounded by some polynomial.

\begin{lemma}[Sauer-Shelah lemma] \label{sauer_shelah}
	Suppose $\mathcal{S}$ has VC-dimension at most $k$. Then
	\[ \pi_\mathcal{S}(n) \leq \sum_{i=0}^{k} \binom{n}{i} \]
	for all $n \geq k$. In particular, there is a constant $C$ such that $\pi_\mathcal{S}(n) \leq C n^k$ for all $n$.
\end{lemma}

Another model theoretic property is the tree property of second kind (TP\textsubscript{2}). By \cite[Proposition 5.31]{simon} every theory that has TP\textsubscript{2} also has IP. Like the independence property, TP\textsubscript{2} can be formulated as a combinatorial condition on formulas.

\begin{definition}
	Suppose $\mathcal{L}$ is a language, $T$ is an $\mathcal{L}$-theory and $\phi(x, y)$ is an $\mathcal{L}$-formula.
	$\phi(x, y)$ has TP\textsubscript{2} if there is some $k \in \mathbb{N}$, some model $M \models \mathcal{T}$, and a family of constants $(a_{ij}: i,j \in \mathbb{N})$ such that
	\begin{itemize}
		\item[(a)] for each $i \in \mathbb{N}$ the formulas $(\phi(x, a_{ij}) : j \in \mathbb{N})$ are $k$-inconsistent, i.e. any conjunction of $k$ distinct such formulas is inconsistent, and
		\item[(b)] for any function $f : \mathbb{N} \rightarrow \mathbb{N}$, the set $\{ \phi(x,a_{i,f(i)}) : i \in \mathbb{N} \}$ is consistent.
	\end{itemize}
	The theory $T$ is NTP\textsubscript{2} if no formula in $T$ has TP\textsubscript{2}.
\end{definition}

We say that a model $M$ has NTP\textsubscript{2} if its theory $\Th(M)$ has NTP\textsubscript{2}.
By the compactness theorem a formula has TP\textsubscript{2} if and only if conditions (a) and (b) are satisfied for arbitrary large finite subsets of $\N$.

\section{Interpretablility of uniform pro-$p$ groups}
We introduce the structure $\Z_p^\text{an}$ which was studied by du Sautoy in \cite{sautoy}.
It follows from results by van den Dries, Haskell, and Macpherson in \cite{haskell} that the structure $\Z_p^\text{an}$ has NIP.
Let $X = (X_1, \dots X_m)$ be a tuple of variables. Given $i = (i_1, \dots i_m) \in \N^m$ we write $X^i = X^{i_1} \cdots X^{i_m}$ and $|i| = i_1 + \dots + i_m$.
Then
\[ \Q_p[[X]] = \{ \sum_{i \in \N^m} a_i X^i : a_i \in \Q_p \} \]
is the ring of formal power series with coefficients in $\Q_p$. $\Z_p[[X]]$ is the subring consisting of all formal power series
with coefficients in $\Z_p$. Let
\[ \Q_p\{X\} = \{ \sum_{i \in \N^m} a_i X^i : a_i \in \Q_p, |a_i | \rightarrow 0 \text{ as } |i| \rightarrow \infty \} \]
and set $\Z_p\{X\} = \Q_p\{X\} \cap \Z_p[[X]]$.

\begin{definition}
	The language $\mathcal{L}_D^\text{an}$ consists of
	\begin{itemize}
		\item[(a)] for every $m$ and every $F(X) \in \Z_p\{X\}$ an $m$-ary function symbol $F$,
		\item[(b)] a binary function symbol $D$, and
		\item[(c)] a unary relation symbol $P_n$ for every $n > 0$.
	\end{itemize}
\end{definition}

We can view $\Z_p$ as an $\mathcal{L}_D^\text{an}$-structure: Let $\nu$ be the usual valuation on $\Z_p$. The relation $P_n$ is given by the set of nonzero $n$-th powers,
the binary function $D$ is given by
\[ D(x,y) = \begin{cases} x/y & \text{if } \nu(x) \geq \nu(y) \text{ and } y \neq 0 \\ 0 & \text{otherwise} \end{cases} \]
and each function $F$ is the function induced by the corresponding power series $F(X)$.
Set $T_D^\text{an} = \Th(\Z_p^\text{an} )$.

Haskell and Macpherson give a definition for P-minimal theories in \cite{pminimal}. By Proposition 7.1 of \cite{pminimal} every P-minimal theory has NIP. Theorem A of \cite{haskell} essentially states that the theory of $\Z_p^\text{an}$ is P-minimal. Therefore we have the following:

\begin{theorem}[Theorem 3.1 of \cite{tent}]
	The theory $T_D^\text{an}$ has NIP.
\end{theorem}

We will need the following two facts. The first uses topological compactness of $\Z_p$, the second is an application of Hensel's lemma.

\begin{lemma}[Lemma 1.9 of \cite{sautoy}]
	Every analytic function is definable in $\Z_p^\text{an}$.
\end{lemma}

\begin{lemma}[Lemma 1.10 of \cite{sautoy}]
	The binary relation $\nu(x) \leq \nu(y)$ is definable in $\Z_p^\text{an}$.
\end{lemma}

Let $G$ be a uniform pro-$p$ group and put $d = \gpdim(G)$. Fix a topologically generating set $\{ x_1, \dots x_d \}$. As seen in Section 2.5, the map $\lambda \mapsto x_1^\lambda$ induces a homeomorphism $\Z_p^d \cong G$.
By Section 3.2 this is an analytic structure on $G$ which makes $G$ into a $p$-adic analytic group.
In particular, the group operations are analytic and hence definable by the previous lemma.
By Theorem~\ref{cont_analytic} the $\Z_p$-action on $G$ is also analytic and thus definable in $\Z_p^\text{an}$.

\begin{proposition} \label{good_base_def}
	Suppose $G$ is a uniform pro-$p$ group. Then the set of all good bases is definable in $\Z_P^\text{an}$.
\end{proposition}
\begin{proof}
	Set $d = \gpdim(G)$, let $(h_1, \dots h_d) \in G^p$ and write $h_i = x(\lambda_i)$. We will show that the conditions (a) to (d) in Lemma~\ref{good_basis_lemma} are definable in $\Z_p^\text{an}$.
	
	(a) We have $\omega(h_i) = \min\{ \nu(\lambda^i_1), \dots \nu(\lambda^i_d) \} +1$ by Proposition~\ref{omega_x_lambda} and therefore $\omega(h_i) \leq \omega(h_j)$ if and only if
	\[ \min\{ \nu(\lambda^i_1), \dots \nu(\lambda^i_d) \} +1 \leq \min\{ \nu(\lambda^j_1), \dots \nu(\lambda^j_d) \} +1. \]
	But this is clearly definable as $\nu(x) \leq \nu(y)$ is definable by the above lemma.
	
	(b) Note that $h_i = 0$ if and only if $\lambda^i = 0$.
	
	(c) By the results of Section 3.2 the group operation and exponentiation with elements from $\Z_p$ are analytic and hence definable in $\Z_p^\text{an}$.
	Note that $\{ h_1^{\lambda_1} \cdots h_d^{\lambda_d} : \lambda_i \in \Z_p\}$ is definable from the parameters $h_1, \dots h_d$.
	
	(d) Follows easily from definability of $\nu(x) \leq \nu(y)$ and definability of the  $\Z_p$-action.
\end{proof}

\begin{definition}
	The language $\mathcal{L}_\text{prof}$ is a two-sorted language with sorts $G$ and $I$ with the group language on $G$, a partial order $\leq$ on $I$, and a relation $K \subseteq G \times I$. If $G$ is a profinite group together with a family $\{ K_i : i \in I \}$
	of open subgroups, which is a neighborhood basis at $1 \in G$, then $\mathcal{G} = (G,I)$ becomes an $\mathcal{L}_\text{prof}$-structure by defining
	$K(G,i) = K_i$ for each $i \in I$ and $i \leq j \iff K_i \supseteq K_j$.
	We say that $\mathcal{G}$ is a \textit{full} profinite group, if the family $\{K_i : i \in I\}$ consists of all open subgroups of $G$.
\end{definition}

We have already seen that $G$ is definable in $\Z_p^\text{an}$ as a group. By Lemma~\ref{good_base_exist} every open subgroup of $G$ admits a good basis. The set of all good bases is definable by Proposition~\ref{good_base_def}. By Lemma~\ref{good_basis_lemma} the open subgroups are uniformly definable from good bases.
Moreover, the relation saying that two good bases generate the same open subgroup is definable.
Therefore we obtain the following:

\begin{theorem}[Proposition 1.2 of \cite{tent}] \label{full_profinite_interpret}
	Let $G$ be a uniform pro-$p$ group. Then the full profinite group $\mathcal{G} = (G,I)$ is interpretable in $\Z_p^\text{an}$ and hence has NIP.
\end{theorem}

\section{Classification of full profinite NIP groups}
	We have seen that examples of full profinite NIP groups are given by compact $p$-adic analytic groups. It turns out that there are essentially no further examples. Every full profinite NIP group
	is virtually a finite direct product of compact $p$-adic analytic groups.
	
\begin{theorem}[Theorem 1.1 of \cite{tent}] \label{main_theorem}
	Let $\mathcal{G} = (G,I)$ be a full profinite group. Then the following are equivalent:
	\begin{itemize}
		\item[(a)] $G$ has an open normal subgroup $N \trianglelefteq_o G$ such that $N = P_1 \times \dots \times P_t$ is a direct product of compact $p_i$-adic analytic groups $P_i$, where the $p_i$ are different primes.
		\item[(b)] $\Th(\mathcal{G})$ is NIP.
		\item[(c)] $\Th(\mathcal{G})$ is NTP\textsubscript{2}.
	\end{itemize}
\end{theorem}

We just sketch the proof. Its details are contained in \cite{tent}:

\medskip
\noindent
{\bf (a) implies (b)}
Suppose $G$ is a group and $N \trianglelefteq_o G$ is an open normal subgroup of $G$ such that
	\[ N = P_1 \times \dots \times P_t \]
	is a Cartesian product of compact $p_i$-adic analytic groups $P_i$, where the $p_i$ are different primes.
	We have to show that $G$ has NIP when viewed as a full profinite group.

By Lazard's Theorem, every $P_i$ has an uniformly powerful pro-$p_i$ normal subgroup of finite index. Hence we may assume that each $P_i$ is uniformly powerful and hence  interpretable in $\mathbb{Z}_{p_i}^\mathrm{an}$ as a full profinite group.
Let $M$  be the disjoint union of the structures $\mathbb{Z}_{p_i}^\mathrm{an}$ in the language that is the disjoint union of their languages.

Naming constants for the finite group $F = G/N$, it is left to show that $G$ can be described in terms of
$N$, $F$ and functions $F \times F \rightarrow N$ and $F \rightarrow \Aut(G)$,
and that the open subgroups of $G$ are uniformly definable in $(M,F)$, proving (b).

\bigskip

The following combinatorial lemma will be used repeatedly in the proof that either of (b) or (c) implies (a).

\begin{lemma}[Lemma 4.3 of \cite{tent}] \label{ntp2_lemma}
	Suppose $G$ is an $\emptyset$-definable group in a structure with NTP\textsubscript{2} theory and $\psi(x,y)$ is a formula implying $x \in G$.
	Then there is a constant $k = k_\psi$, depending only on $\psi$, such that whenever $H \leq G$ is a subgroup, $\pi : H \twoheadrightarrow \prod_{i \in J} T_i$ is an epimorphism from $H$ onto a product of groups and for each
	$j \in J$ there are $\overline{R_j} \leq G$ and $R_j < T_j$ such that $\overline{R_j} \cap H = \pi_j^{-1}(R_j)$ and finite intersections of the $\overline{R_j}$ are uniformly definable by instances of $\psi$, then $|J| \leq k$. If $G$ is a NIP group, then it suffices that the $\overline{R_j}$ are definable.
\end{lemma}

Since we are working in full groups, it is clear from this lemma that the distinction between NIP and NTP$_2$ groups disappears in this setting.

\medskip
\noindent
{\bf (b) or (c) imply (a) }

Using the fact that if all Sylow subgroups of a finite group $G$ can be generated by $d$ elements, then $\gpdim(G) \leq d+1$ (proved independently by Lucchini in \cite{lucchini} and by Guralnick in \cite{guralnick}) we first show that if
 $\mathcal{G} = (G,I)$ is a full profinite group with NTP\textsubscript{2} theory, then $G$ has finite rank.

By Theorem~\ref{reid}, the group $G$ has closed normal subgroups $N \leq A \leq G$ such that $N = F(G)$ is the pro-Fitting subgroup, $A/N$ is finitely generated abelian and $G/A$ is finite. In particular, $N$ is pronilpotent of finite rank.

We may assume $G = A$. The quotient $G/N$ is abelian, hence pronilpotent and thus the Cartesian product of its Sylow subgroup.
Similarly,  the pronilpotent group $N$ is a Cartesian product of its Sylow subgroups. So write
\[G/N = \prod_{l \in L} Q_l\mbox{\ \ and \ }N = \prod_{j \in J}P_j \] where  the $Q_l$ are $r_l$-Sylow subgroups for different primes $r_l$ and the $P_j$ are $p_j$-Sylow subgroups for different primes  $p_j$.
Using Lemma~\ref{ntp2_lemma} one shows that the sets $L$  and $J$ are finite.

Let $\pi_N : G \twoheadrightarrow G/N$ be the natural projection. Write $\overline{U} = \pi_N^{-1}(U)$ where $U$ is the direct product of finitely many Sylow subgroups of $G/N$ with respect to primes different from the primes appearing in $N$. As $|N|$ and $|U|$ are coprime, it follows that $N$ is a Hall subgroup in $\overline{U}$ and we can use the Schur-Zassenhaus theorem to find a complement $V$ of $N$ in $\overline{U}$, i.e. $\overline{U} = N \rtimes V$. By Theorem~\ref{aut_pro_p}, $V$ induces a finite group of automorphisms on $N$
and we eventually conclude that $G = N = P_1 \times \dots \times P_t$. As $G$ has finite rank, so has each $P_i$. Hence each $P_i$ is a compact $p_i$-adic analytic group.

\section{NIP groups and polynomial subgroup growth}
We show here that every family of uniformly definable subgroups of a NIP group satisfies a polynomial growth condition.
By Lazard's theorem and Lubotzky and Mann's results in \cite{psg}, a pro-$p$ group is compact $p$-adic analytic if and only if it has polynomial subgroup growth, i.e.
there is some polynomial $f(X)$ such that for all $n$ there are at most $f(n)$ many open subgroups of index at most $n$.
This then yields a more immediate proof that a  full pro-$p$ group with NIP theory is $p$-adic analytic.

\begin{theorem}\label{t:NIPpolybd}
Let $T$ be a theory, $M \models T$ a model of $T$, and let $G$ be an $\emptyset$-definable group. Assume further that $\phi(x,y)$ is a formula which has NIP and implies $x \in G$.

Let $\sigma_\phi(n)$ be the number of subgroups $H \leq G$ such that $H$ has index at most $n$ and can be defined by an instance of $\phi$.
Then there is a constant $c$ such that
	\[ \sigma_\phi(n) \leq cn^{\VC(\phi) \VC(\phi^\text{opp})} \]
	for all $n$.
\end{theorem}
\begin{proof}
By the Baldwin-Saxl lemma (\ref{baldwin_saxl}), the intersection of any finite number of such subgroups has index at most $n^{\VC(\phi^\text{opp})}$ and hence $\sigma_\phi(n)$ is finite for all $n$
and therefore does not depend on our choice of $M$.
	
As $\phi(x,y)$ has NIP, the family $\{ \phi(M,b) : b \in M \}$ has finite VC dimension. 
By the Sauer-Shelah lemma (\ref{sauer_shelah}), there is a constant $c$ such that the shatter function for this family satisfies
\[ \pi_\phi(n) = \max \{ |\{X\cap \phi(M,b) : b \in M \}| : X \subseteq M, |X| \leq n \} \leq c n ^{\VC(\phi)} \]
 for all~$n$.
Note that this function does not depend on the model $M$. 
Combining this with the  bounds from the next lemma yields the desired result.
\end{proof}

\begin{lemma}
	We have $\sigma_\phi(n) \leq \pi_\phi(n^{\VC(\phi^\text{opp})})$ for all $n$.
\end{lemma}
\begin{proof}
	Let $n$ be fixed and let $(H_i : i \in I)$ be the family of subgroups $H \leq G$ such that $H$ has index at most $n$ and is definable by some instance of $\phi$.
	Let $C = \bigcap_{i \in I} H_i$. Then $C$ has index at most $n^{\VC(\phi^\text{opp})}$ by the Baldwin-Saxl lemma. Let $X$ be a transversal of $G/C$ and note that
	\[  H_i = H_j \iff X \cap H_i = X \cap H_j. \]
	Therefore we have
	\[ \sigma_\phi(n) \leq |\{ X \cap \phi(M,b) : b \in M \}| \leq \pi_\phi(n^{\VC(\phi^\text{opp})}). \qedhere \]
\end{proof}

Note that we have only assumed that $\phi(x,y)$ has NIP. If $G$ is a pro-$p$ group, then the above corollary, together with Theorem~\ref{main_theorem}, tells us that the full profinite group $(G,I)$ has NIP if and only if the 
formula, which defines the family of open subgroups, has NIP.
We obtain the following combinatorial characterization of $p$-adic analytic pro-$p$ groups:

\begin{corollary}
	Let $G$ be a pro-$p$ group. Then $G$ is compact $p$-adic analytic if and only if the family of all open subgroups has finite VC dimension. 
\end{corollary}

\section{Elementary extensions of profinite groups as two-sorted structures}

In this last section we show that elementary extensions of profinite groups work well in this model theoretic setting, in fact, even in the sense that the original group can be recovered. 

\begin{lemma} \label{elem_ext}
	Let $G$ be a profinite group and let $(N_i: i \in I)$ be a neighborhood basis of the identity consisting of open normal subgroups $N_i \trianglelefteq_o G$.
	Let $\mathcal{G} = (G,I)$ be the corresponding $\mathcal{L}_\text{prof}$-structure and let $\mathcal{G}^* = (G^*, I^*)$ be an elementary extension of $\mathcal{G}$. Then the natural homomorphism 
	\[G \rightarrow G^* / \bigcap_{i \in I} N_i^*\]
	 is an isomorphism.
\end{lemma}
\begin{proof}
	Note that $|G:N_i| = |G^*:N_i^*|$ holds for all $i \in I$. The natural maps $G \rightarrow G^*/N_i^*$ are surjective for all $i \in I$.
	Therefore 
	\[f : G \rightarrow \varprojlim_{i \in I} G^*/N_i^*\]
	is surjective. On the other hand we have $\ker f = \bigcap_{i \in I} N_i = 1$ by elementarity, and hence $f$ is an isomorphism. Therefore the natural homomorphism
	\[ g : G^* \rightarrow \varprojlim_{i \in I} G^*/N_i^* \]
	is surjective. We have
	\[ G \cong \varprojlim_{i \in I}G^*/N_i^* \cong G^*/\bigcap_{i \in I}N_i^*. \qedhere \]
\end{proof}

Let $T$ be an NIP theory and let $G$ be a $\emptyset$-definable group. Then $G^0$ denotes the intersection of all $\emptyset$-definable subgroups of finite index.
By Baldwin-Saxl, the quotient $G/G^0$ is independent from the choice of a (sufficiently saturated) model (see e.g. \cite[Section 8.1.2]{simon}). If $\mathcal{G} = (G,I)$ is a full profinite NIP group then the finite index subgroups are open (by Theorem~\ref{anderson} and Theorem~\ref{main_theorem}) and hence uniformly definable. The normal subgroups of finite index are also uniformly definable and
every subgroup of finite index contains a normal subgroup of finite index. Therefore we can apply the previous lemma to obtain the following (see \cite[Remark 5.5]{tent}).

\begin{corollary}
	Suppose $\mathcal{G} = (G,I)$ is a full profinite NIP group. Then the invariant quotient $G/G^0$ is isomorphic to $G$.
\end{corollary}

In \cite{mariano}, Mariano and Miraglia showed that profinite $\mathcal{L}$-structures are retracts of ultraproducts of finite $\mathcal{L}$-structures. We can use Lemma~\ref{elem_ext} to give a short proof in the case of profinite groups.

Let $G$ be an infinite profinite group and let $(N_i: i \in I)$ be a neighborhood basis of the identity consisting of open normal subgroups $N_i \trianglelefteq_o G$. For $t \in I$ set $B_t = \{ j \in I : t \leq j \}$. Then $\mathcal{F}_\geq = \{ X \subseteq I : \exists t \in I \, B_t \subseteq X \}$ is a filter on $I$ and contains every cofinite subset of $I$.

\begin{lemma}
	Let $G$ and $(N_i:i \in I )$ be as above and let $\mathcal{G} = (G,I)$ be the corresponding $\mathcal{L}_\text{prof}$-structure. Let $\mathcal{U} \supseteq \mathcal{F}_\geq$ be an ultrafilter
	and let $\mathcal{G}^* = (G^*,I^*)$ be the ultrapower of $\mathcal{G}$ with respect tu $\mathcal{U}$.
	Let $\alpha$ be the equivalence class $\alpha = [(i:i \in I)] \in I^*$.
	Then $N_\alpha^* \subseteq \bigcap_{i \in I} N_i^*$ and the composition
	\[ G \rightarrow G^* \rightarrow G^*/N_\alpha^* \rightarrow G^* / \bigcap_{i \in I} N_i^* \]
	is an isomorphism.
\end{lemma}
\begin{proof}
	Let $[(g_i)] \in N_\alpha$, i.e. $\{i \in I : g_i \in N_i \} \in \mathcal{U}$. For $t \in I$ we have
	$ \{ i \in I : g_i \in N_i \} \cap B_t \in \mathcal{U}$. As $i \ge t \iff N_i \subseteq N_t$, it follows $\{ i \in I : g_i \in N_t \} \in \mathcal{U}$ and thus
	$[(g_i)] \in N_t^*$. Hence  $N_\alpha^* \subseteq \bigcap_{i \in I} N_i^*$ and by Lemma~\ref{elem_ext} the natural map $G \rightarrow G^* / \bigcap_{i \in I} N_i^*$ is an isomorphism.
\end{proof}

Note that the natural map $G^* \rightarrow \prod (G/N_i) / \mathcal{U}, [(g_i)] \mapsto [(g_iN_i)]$ induces an isomorphism $G^*/N_\alpha^* \cong \prod (G/N_i)/\mathcal{U}$.
In particular, $\prod (G/N_i) / \mathcal{U}$ is the semidirect product
\[ \prod (G/N_i) / \mathcal{U} = (\bigcap_{i \in I} N_i^*) / N_\alpha^* \rtimes G. \]

The following proposition is the first result in \cite[Remark 5.5]{tent}. It uses \cite[Lemma 3.2]{pseudofinite} and an observation on externally definable sets in \cite[Remark 5.5]{tent}.

\begin{proposition} \label{prof_pseudo_NIP}
	Let $G$ be a profinite group and let $(N_i : i \in I)$ be a neighborhood basis of the identity consisting of open normal subgroups $N_i \trianglelefteq_o G$.
	Let $\mathcal{G} = (G,I)$ be the corresponding $\mathcal{L}_\text{prof}$-structure.
	For each $t \in I$ the group $G/N_t$ naturally becomes an $\mathcal{L}_\text{prof}$-structure
	$\mathcal{G}_t = (G/N_t,I_{\leq t})$. Let $\mathcal{U} \supset \mathcal{F}_\geq$ be an ultrafilter.
	Then $\mathcal{G}$ has NIP if and only if $\prod_{t \in I}\mathcal{G}_t / \mathcal{U}$ has NIP.
\end{proposition}
\begin{proof}
	If $\mathcal{G}$ has NIP  then 
	the structures $(\mathcal{G}_t:t \in I)$ are uniformly interpretable in $\mathcal{G}$
	and hence $\prod_{t \in I}\mathcal{G}_t / \mathcal{U}$ has NIP by \cite[Lemma 3.2]{pseudofinite}.
	
	Now assume $(H,J) = \prod_{t \in I}\mathcal{G}_t / \mathcal{U}$ has NIP and put
	\[ J_0 = \{ j \in J : N_j \text{ has finite index in } H \}. \]
	Given $j,j' \in J_0$ there is $t \in J_0$ such that $N_j$ and $N_{j'}$ both contain $N_t$.
	Hence there is $t^*$ in some elementary extension such that
	\[ J_0 = \{ j \in J : N_j \supset N_{t^*} \}. \]
	Hence $J_0$ is externally definable in the sense of \cite[Definition 3.8]{simon}. By \cite[Corollary 2.24]{simon} the 		expansion of an NIP structure by externally definable sets still has NIP.
	
	If $j = [(j_k)]$ is an element of $J_0$ then there is $i \in I$ such that $j_k = i$ for almost all $k$.
	By the previous lemma the structure $\mathcal{G}$ is interpretable in $(H,J,J_0)$ and hence has NIP.
\end{proof}

\begin{lemma}
	Let $H$ be a group and let $(N_i : i \in I)$ be a family of normal subgroups of finite index such that
	$\forall i,j \exists k : N_k \leq N_i \cap N_j$.
	We view $H$ as an $\mathcal{L}_\text{prof}$-structure $\mathcal{H} = (H,I)$. Let $f_j : \varprojlim H/N_i \rightarrow H/N_j$ be the projection maps. Then $\{ \ker f_j : j \in I \}$ is a neighborhood basis at the identity. Therefore we may view $\varprojlim H/N_i$ as an $\mathcal{L}_\text{prof}$-structure $(\varprojlim H/N_i, I)$.
	
	Let $\mathcal{H}^* = (H^*,I^*)$ be an $|I|^+$-saturated elementary extension of $\mathcal{H}$. Then 
	\[ \mathcal{H}^* / \bigcap_{i \in I} N_i^* \cong \varprojlim_{i \in I} H/N_i \]
	and $(N_j^* / \bigcap_{i \in I} N_i^* : j \in I )$ is a neighborhood basis for the identity consisting of open normal subgroups.
\end{lemma}
\begin{proof}
	By elementarity, we have $|H^*:N_i^*| = |H:N_i|$ for all $i \in I$. Using this and elementarity, it is easy to see that
	\[ \varprojlim_{i \in I} H^*/N_i^* = \varprojlim_{i \in I} H/N_i. \]
	Now write
	\[ \varprojlim_{i \in I} H^*/N_i^* = \{ (g_iN_i^*)_i : \forall i \geq j: g_iN_j^* = g_jN_j^* \} \]
	and let $f : H^* \rightarrow  \varprojlim_{i \in I} H^*/N_i^*, g \mapsto (gN_i^*)_i$ be the natural homomorphism.
	Clearly $\ker f = \bigcap_{i \in I} N_i^*$. It remains to show that $f$ is surjective.
	
	Fix $(g_iN_i)_i \in \varprojlim_{i \in I} H^*/N_i^*$ and consider the partial type
	\[ \Sigma(x) = \{ x \in g_iN_i^* : i \in I \}. \]
	Given $I_o \subseteq I$ finite, there is $j \in I$ such that
	$N_j \leq \bigcap_{i \in I_0} N_i$. Then $g_iN_j = g_{i'}N_j$ for all $i,i' \in I_0$.
	Hence $\Sigma(x)$ is finitely satisfiable and as $\mathcal{H}^*$ is $|I|^+$-saturated, there exists $g \in H^*$ such that $g \in g_iN_i^*$ for all $i \in I$ and hence $f(g) = (g_iN_i^*)$.
	
	The above isomorphism maps $N_j^*/\bigcap_{i \in I} N_i^*$ to $\ker f_j$ and hence the family $(N_j^*/\bigcap_{i \in I} N_i^* : j \in I)$ is a neighborhood basis for the identity consisting of open normal subgroups.
\end{proof}

\begin{lemma}
	Let $\mathcal{H} = (H,I)$ be as in the previous lemma.
	If $\mathcal{H}$ has NIP then the profinite group $(\varprojlim_{i \in I} H/N_i, I)$ has NIP.
\end{lemma}
\begin{proof}
	As $\mathcal{H}$ has NIP, there are only finitely many normal subgroups $N_i$ of index at most $n$ for all $n$.
	
	Fix an $\aleph_1$-saturated elementary extension $\mathcal{H}^* =(H^*,I^*)$ of $\mathcal{H}$. Note that $|H^*:N_i^*|$ must be infinite for all $i \in I^* \setminus I$.
	By a similar argument as in \ref{prof_pseudo_NIP} the set $I \subseteq I^*$ is externally definable.
	
	In particular, the structure $(H^*,I^*,I)$, where $I$ is a unary predicate, has NIP.
	
	By the previous lemma, the structure $(\varprojlim_{i \in I} H/N_i, I)$ is interpretable in the NIP structure $(H^*,I^*,I)$ and thus has NIP.
\end{proof}

The previous lemma allows us to study families of uniformly definable subgroups of finite index in NIP groups by considering certain profinite NIP groups in the language $\mathcal{L}_\text{prof}$.
As an application we obtain the following generalization of \cite[Proposition 5.1]{tent}:

\begin{theorem}
	Let $G$ be an NIP group and let $\phi(x,y)$ be a formula. Let $(N_i:i \in I)$ be the family of all normal subgroups of finite index which are definable by an instance of $\phi$. If this family is infinite then there is a finite subset $I_0 \subseteq I$ such that $\bigcap_{j \in I_0}N_j / (\bigcap_{j \in J_0}N_j \cap N_i)$ is solvable for all $i \in I$.
\end{theorem}
\begin{proof}
	By the Baldwin-Saxl lemma we may assume that $(N_i : i \in I)$ is closed under finite intersections.
	Now consider the corresponding $\mathcal{L}_\text{prof}$-structure $(G,I)$. 
	
	We aim to show that $(G,I)$ has NIP. The set $\{ b : \phi(G,b) \text{ is a normal subgroup of } G \}$ is definable. Since we assume the family $(N_i:i \in I)$ to be closed under finite intersections, the set
	\[ J = \{ b : \phi(G,b) \text{ is a normal subgroup of finite index in } G \} \]
	is externally definable. Therefore $(G,I)$ has NIP since it is interpretable after naming $J$.
	
	By the previous lemma the profinite group $(\varprojlim_{i \in I} G/N_i, I)$ has NIP. Therefore it is virtually prosolvable by \cite[Proposition 5.1]{tent}.
\end{proof}

\bibliography{MMMsurvey}
\bibliographystyle{plain}

\end{document}